\documentclass[reqno]{amsart}
\usepackage[utf8]{inputenc}
\usepackage{graphicx}
\usepackage{svg}

\usepackage{color}
\usepackage{refstyle}
\usepackage{calc}
\usepackage{amsmath}
\usepackage{amsthm}
\usepackage{amssymb}
\usepackage{setspace}
\PassOptionsToPackage{normalem}{ulem}
\usepackage{ulem}
\usepackage[unicode=true,
colorlinks=true]{hyperref}

\makeatletter

\AtBeginDocument{}
\AtBeginDocument{\providecommand\Lemref[1]{\ref{Lem:#1}}}
\AtBeginDocument{\providecommand\Thmref[1]{\ref{Thm:#1}}}
\AtBeginDocument{\providecommand\Blackboxref[1]{\ref{proposition:#1}}}

\RS@ifundefined{subsecref}
  {\newref{subsec}{name = \RSsectxt}}
  {}
\RS@ifundefined{thmref}
  {\def\RSthmtxt{theorem~}\newref{thm}{name = \RSthmtxt}}
  {}
\RS@ifundefined{lemref}
  {\def\RSlemtxt{lemma~}\newref{lem}{name = \RSlemtxt}}
  {}

\theoremstyle{plain}
\newtheorem{thm}{\protect\theoremname}
\theoremstyle{plain}
\newtheorem{lem}[thm]{\protect\lemmaname}

\theoremstyle{plain}
\newtheorem{nlemtemp}{Lemma}
\newenvironment{nlem}[1]
  {\renewcommand\thenlemtemp{#1}\nlemtemp}
  {\endnlemtemp}

\theoremstyle{remark}
\newtheorem{rem}[thm]{\protect\remarkname}
\theoremstyle{remark}
\newtheorem*{rem*}{\protect\remarkname}
\theoremstyle{plain}
\newtheorem{proposition}[thm]{\protect\propname}
\theoremstyle{plain}

\theoremstyle{remark}
\newtheorem{claim}{\protect\claimname}

\@ifundefined{date}{}{\date{}}

\usepackage{amsfonts}
\usepackage{amsthm}
\allowdisplaybreaks
\usepackage{cases}

\DeclareMathOperator{\vol}{vol}

\DeclareMathOperator{\Div}{div}
\DeclareMathOperator{\Ric}{Ric}
\DeclareMathOperator{\Def}{Def}
\DeclareMathOperator{\curl}{curl}
\DeclareMathOperator{\grad}{grad}
\DeclareMathOperator{\Range}{Range}

\DeclareMathOperator{\tr}{tr}

\usepackage{refstyle}
\newref{cor}{name = corollary~, names = corollaries~, Name = Corollary~, Names = Corollaries~}
\newref{conj}{name= conjecture~, names = conjectures~, Name = Conjecture~, Names = Conjectures~}
\newref{exa}{name= example~, names = examples~, Name = Example~, Names = Examples~}
\newref{rem}{name= remark~, names = remarks~, Name = Remark~, Names = Remarks~} 
\newref{thm}{name= theorem~, names = theorems~, Name = Theorem~, Names = Theorems~}
\newref{lem}{name= lemma~, names = lemmas~, Name = Lemma~, Names = Lemmas~}
\newref{def}{name= definition~, names = definitions~, Name = Definition~, Names = Definitions~}
\newref{fact}{name= fact~, names = facts~, Name = Fact~, Names = Facts~}
\newref{prop}{name= proposition~, names = propositions~, Name = Proposition~, Names = Propositions~}
\newref{sec}{name= section~, names = sections~, Name = Section~, Names = Sections~}
\newref{subsec}{name= subsection~, names = subsections~, Name = Subsection~, Names = Subsections~}

\def\XXint#1#2#3{{\setbox0=\hbox{$#1{#2#3}{\int}$ }
\vcenter{\hbox{$#2#3$ }}\kern-.6\wd0}}

\newcommand\iso{\xrightarrow{
   \,\smash{\raisebox{-0.65ex}{\ensuremath{\scriptstyle\sim}}}\,}}

\usepackage{tensind}

\hypersetup{citecolor=blue,linkcolor=blue}

\newcommand\restr[2]{{
  \left.\kern-\nulldelimiterspace 
  #1 
  \vphantom{\big|} 
  \right|_{#2} 
  }}

\makeatother

\providecommand{\propname}{Proposition}
\providecommand{\claimname}{Claim}
\providecommand{\factname}{Fact}
\providecommand{\lemmaname}{Lemma}
\providecommand{\remarkname}{Remark}
\providecommand{\theoremname}{Theorem}

\allowdisplaybreaks

\begin{document}
\title[Geometric trapping for NSE on 2-manifolds]{A geometric trapping approach to global regularity 
for 2D Navier-Stokes on manifolds}
\author{Aynur Bulut}
\address[A. Bulut]{Department of Mathematics, Louisiana State University}
\email{aynurbulut@lsu.edu}
\author{Khang Manh Huynh}
\address[K.M. Huynh]{Department of Mathematics, University of California, Los Angeles}
\email{hmkhang24@math.ucla.edu}
\maketitle
\begin{abstract}
In this paper, we use frequency decomposition techniques to give a direct proof of global
existence and regularity for the Navier-Stokes equations on two-dimensional
Riemannian manifolds without boundary.  Our techniques are inspired by an approach 
of Mattingly and Sinai \cite{mattinglyElementaryProofExistence1999} which was developed in the 
context of periodic boundary conditions on a flat background, and which is based on a maximum principle for
Fourier coefficients.

The extension to general manifolds requires several new ideas, connected to the 
less favorable spectral localization properties in our setting.  Our arguments make use of 
frequency projection operators, multilinear estimates that originated in the study of the 
non-linear Schr\"odinger equation, and ideas from microlocal analysis.
\end{abstract}

\section{Introduction}

Let $(M,g)$ be a closed, oriented, connected, compact smooth two-dimensional
Riemannian manifold, and let $\mathfrak{X}(M)$ denote the space of smooth vector
fields on $M$.  We consider the incompressible Navier-Stokes equations 
on $M$, with viscosity coefficient $\nu>0$,
\begin{equation}
\left\{ \begin{array}{rll}
\partial_{t}U+\Div\left(U\otimes U\right)-\nu\Delta_{M}U & =-\grad p & \text{ in }M\\
\Div U & =0 & \text{ in }M
\end{array}\right.,\label{eq:NavierStokes}
\end{equation}
with initial data $$U_0\in \mathfrak{X}(M),$$ where $I\subset\mathbb{R}$ is an open interval, 
and where $U:I\to\mathfrak{X}(M)$ and $p:I\times M\to\mathbb{R}$ represent the velocity
and pressure of the fluid, respectively.  Here, the operator $\Delta_{M}$ is 
any choice of Laplacian defined on vector fields on $M$, discussed below.

The theory of two-dimensional fluid flows on flat spaces is well-developed, and a 
variety of global regularity results are well-known.  This includes results on the 
whole space $\mathbb{R}^2$, on smooth bounded domains $\Omega\subset\mathbb{R}^2$, 
and on the square $(0,1)^2$ with periodic boundary conditions, which corresponds 
to the flat torus $\mathbb{T}^2$.  Important results in this direction are due to 
Ladyzhenskaya \cite{ladyzhenskaya}, and Fujita-Kato \cite{fujitaNavierStokesInitialValue1964}, 
with the latter analysis being based on estimates for the heat semigroup.  

In \cite{mattinglyElementaryProofExistence1999}, Mattingly and Sinai give an elementary 
proof of regularity (and in fact analyticity) for the periodic setting -- see also the 
references cited in \cite{mattinglyElementaryProofExistence1999} for a summary of other 
related results.  The technique in \cite{mattinglyElementaryProofExistence1999} 
works directly with sequences of Fourier coefficients.  They establish a priori bounds 
for the two-dimensional flow by appealing to a Galerkin method and invoking a variant 
of the maximum principle applied to the system of ODEs for the Fourier coefficients.

In this paper, we develop and extend this {\it geometric trapping} method to the case of the
Navier-Stokes system posed on a general manifold $M$ satisfying conditions as above.

The study of fluid equations such as (\ref{eq:NavierStokes}) posed on manifolds has a 
long history.  In addition to the physical motivation, where fluid models posed on 
surfaces such as the sphere emerge naturally as we consider atmospheric models of the 
Earth, the PDEs of fluids are intimately tied to geometry.  The interplay 
between geometry and analysis arising in the study of fluid dynamics has inspired 
new developments in many directions.  This includes applications to Hodge theory, 
the Euler-Arnold equation, Leray's sheaf theory, Killing vector fields, and other 
areas.  We refer interested readers to \cite{arnoldGeometrieDifferentielleGroupes1966, ebinGroupsDiffeomorphismsMotion1970,caoNavierStokesEquationsRotating1999, caoRammahaTiti2000, kobayashiNavierStokesEquations2008,mitreaNavierStokesEquationsLipschitz2001, taylorPartialDifferentialEquations2011,huynhHodgetheoreticAnalysisManifolds2019} and the references contained in these works.

Before proceeding, we elaborate on the choice of the vector Laplacian 
$\Delta_M$.  Due to the influence of curvature, there are essentially three canonical choices for 
the vector Laplacian, 
\begin{itemize}
\item the Hodge-Laplacian $\Delta_{H}=-\left(d\delta+\delta d\right)$,
which is defined on differential forms, and then extended to vector
fields by the musical isomorphism,
\item the connection Laplacian (or {\it Bochner Laplacian})
$\Delta_{B}T:=\tr\left(\nabla^{2}T\right)=\nabla_{i}\nabla^{i}T$
for any tensor $T$ (note that, by the Weitzenbock formula, which we recall in Appendix \ref{appendix:notation}, we have $\Delta_{B}X=\Delta_{H}X+\Ric(X)$ for all smooth
vector fields $X$ on $M$), and
\item the deformation Laplacian $\Delta_{D}X=-2\Def^{*}\Def X=2\Div\Def X$
where $\left(\Def X\right)^{ij}=\frac{1}{2}\left(\nabla^{i}X^{j}+\nabla^{j}X^{i}\right)$
for $X\in\mathfrak{X}(M)$. Then, for all smooth vector fields $X$, $\Delta_{D}X=\Delta_{H}X+2\Ric(X)+\grad\Div X$.
Since $\Div U=0$ in the Navier-Stokes equation, we can treat $\Delta_{D}$
as $\Delta_{H}+2\Ric$ for the incompressible Navier-Stokes equation.
\end{itemize}

Each of the operators $\Delta_H$, $\Delta_B$, and $\Delta_D$ have the same principal symbol (or leading
terms), and so our treatment is largely indepedent of the specific choice of $\Delta_M$. In the context of fluid models 
on manifolds, the Hodge Laplacian was used in \cite{caoNavierStokesEquationsRotating1999,kobayashiNavierStokesEquations2008},
while the deformation Laplacian was preferred in more recent works
such as \cite{taylorPartialDifferentialEquations2011,chanFormulationNavierStokesEquations2017,mitreaNavierStokesEquationsLipschitz2001,pruessNavierStokesEquationsSurfaces2020}.  We use the convention that all three operators are negative definite, to 
be consistent with the scalar Laplacian (that is, the Laplace-Beltrami operator 
$\Delta f=\Delta_{H}f=\Div\grad f=\nabla^{i}\nabla_{i}f$).

We are now ready to state our main result.

\begin{thm}
\label{thm:general-manifold}
Let $(M,g)$ be a closed, oriented, connected, compact smooth two-dimensional
Riemannian manifold, and let $\Delta_M$ be any of the vector Laplacian
operators $\Delta_H$, $\Delta_B$, or $\Delta_D$ on $M$.

Suppose that $U_0\in \mathfrak{X}(M)$.  Then there exists a unique global-in-time 
smooth solution $U:\mathbb{R}\rightarrow \mathfrak{X}(M)$ to (\ref{eq:NavierStokes})
with $U(0)=U_0$.
\end{thm}

As we mentioned above, to prove Theorem $\ref{thm:general-manifold}$, we will extend and 
develop the geometric trapping ideas that originated in the setting of the two-dimensional
torus in  
\cite{mattinglyElementaryProofExistence1999}.  In fact, our methods 
give sharper information about the regularity of solutions than what we have stated; we choose 
to state the basic smoothness claim to focus on the essential aspects of the argument.

In our setting, a number of subtleties arise that require new ideas beyond the treatment in 
\cite{mattinglyElementaryProofExistence1999}, even in the case of the sphere 
$S^2\subset\mathbb{R}^3$.  To illustrate this, one can ask whether 
replacing $e^{i2\pi\left\langle k,z\right\rangle }$
with spherical harmonics (the eigenfunctions on the sphere) would
extend the result to the sphere.  However, such an approach will not work directly.  
This is because of the poor spectral localization
of products on the sphere (unlike $e^{i2\pi\left\langle k_{1},z\right\rangle }e^{i2\pi\left\langle k_{2},z\right\rangle }=e^{i2\pi\left\langle k_{1}+k_{2},z\right\rangle }$
for the torus). At most, the resulting frequency will lie in a region
defined by triangle inequalities. Moreover, the $L^{2}$ estimate
of the product suffers from an extra factor $\min\left(k_{1},k_{2}\right)^{\frac{1}{4}}$,
which is essentially sharp on the sphere. This will lead to an unacceptable
loss of decay in summing the frequencies. 

Instead, we follow a different approach.  We will group eigenfunctions with the same eigenvalue 
together, and work with 
eigenspace projections instead of Fourier coefficients. We will also replace the 
non-optimal use of H\"older's inequality in the bounds by
multilinear estimates from the theory of non-linear Schr\"odinger equations 
\cite{burqMultilinearEigenfunctionEstimates2005}.  Combining this with a few additional
technical tools, we gain enough decay to obtain geometric trapping in the case when 
$M$ is the sphere $S^2\subset\mathbb{R}^3$.

For general compact manifolds, the situation is even more complicated.
The spectral localization of products is poorer (with no triangle
inequalities), the Ricci tensor is no longer constant, and there can be non-trivial harmonic 1-forms.  To handle the non-triangle regions, we
extend some estimates from \cite{haniGlobalWellposednessCubic2011}, generalizing 
the argument to handle more derivatives as needed in our setting.

To handle the extra terms coming from the Ricci tensor, we use an integration
by parts argument as in the method of stationary phase. To avoid dealing with 
the distribution of eigenvalues on manifolds, we use frequency cutoffs as 
defined by the functional calculus of the Laplacian.  The passage between 
eigenspace projections and frequency cutoffs for multilinear estimates is 
made possible by a Fourier decomposition technique.

\subsection*{Outline of the paper}

In Section 2, we recall our notation and give a preliminary derivation of the precise
formulation of the Navier-Stokes system (\ref{eq:NavierStokes}) that we will use
in our analysis, including the construction of a sequence of Galerkin projections, for 
which we will establish a priori bounds.

The global regularity results of Theorem \ref{thm:general-manifold}
will follow from a priori bounds for this system (independent of the projection);
we establish these in Section \ref{sec:geom-trap}, where we formulate the geometric
trapping construction, and Section \ref{sec:viscous}, which contains the main estimates
that allow us to take full advantage of the diffusive effect of the viscosity term.

\subsection*{Acknowledgments}

We are very grateful to Terence Tao, Yakov Sinai, and Zaher Hani for valuable discussions
during the preparation of this work.

\section{Notation and preliminaries}

In this section, we establish our notation, and derive the main formulation
of the Navier-Stokes system (posed on the manifold $M$) that we will use in our analysis.
In particular, after recalling our notation and introducing a relevant class of
frequency cutoff and projection operators in Sections \ref{sec:notation} and 
\ref{sec:freq_cutoff}, we formulate the system in terms of an equation for the {\it vorticity} (see
Section \ref{sec:vorticity}), and introduce a sequence of Galerkin projections
(see Section \ref{sec:galerkin}), along with some preliminary analysis of the
relevant a priori estimates for the system.

\subsection{Geometric notation \& review}
\label{sec:notation}

Unless mentioned otherwise, the metric $g$ is the Riemannian metric,
and the connection $\nabla$ is the Levi-Civita connection.  We write 
$\left\langle \cdot,\cdot\right\rangle $ to denote the Riemannian
fiber metric for tensor fields on $M$. We also define the dot product
\[
\left\langle \left\langle \sigma,\theta\right\rangle \right\rangle =\int_{M}\left\langle \sigma,\theta\right\rangle \vol
\]
where $\sigma$ and $\theta$ are tensor fields of the same type,
while $\vol$ is the Riemannian volume form. When there is
no possible confusion, we will omit writing $\vol$.

Let $\Omega^1(M)$ denote the space of $1$-forms on $M$.  As usual, for any smooth vector 
field $X\in\mathfrak{X}(M)$, we define $X^{\flat} \in\Omega^{1}(M)$, also denoted by
$\flat X$, by setting $X^{\flat}(Y):=\left\langle X,Y\right\rangle$ for $Y\in \mathfrak{X}(M)$.
This is the so-called {\it musical isomorphism}, which 
identifies $\mathfrak{X}(M)$ with $\Omega^{1}(M)$. Similarly, we define
$\left\langle \alpha^{\sharp},Y\right\rangle =\alpha(Y)$ for any
$\alpha\in\Omega^{1}(M)$ and $Y\in\mathfrak{X}(M)$.

As in \cite{huynhHodgetheoreticAnalysisManifolds2019}, we will
often use {\it Penrose abstract index notation} (cf. \cite[Section 2.4]{waldGeneralRelativity1984}),
where the indices do not correspond to any preferred coordinate system,
but only indicate the types of tensors and how they contract. This
should not be confused with the similar-looking Einstein notation
for local coordinates, or the similar-sounding Penrose graphical
notation.  We collect the usual identities in
differential geometry (proved in \cite{leeManifoldsDifferentialGeometry2009}
and \cite{waldGeneralRelativity1984}), expressed in Penrose notation, in Appendix \ref{appendix:notation}.

We conclude this section by recalling some conventions and common notation used 
throughout the rest of this paper.  We will write $A\lesssim_{x,\neg y}B$ for 
$A\leq CB$, where $C$ is a positive constant depending on $x$ and not $y$. Similarly, 
$A\sim_{x,\neg y}B$ means $A\lesssim_{x,\neg y}B$ and $B\lesssim_{x,\neg y}A$. We
will omit the explicit dependence when it is either not essential or obvious by context.

With $\Omega^k(M)$, $k\geq 1$, denoting the space of $k$-forms on $M$, we recall the usual 
Hodge star operator $\star:\Omega^{k}(M)\iso\Omega^{n-k}(M)$, the exterior derivative 
$d:\Omega^{k}(M)\to\Omega^{k+1}(M)$, the codifferential $\delta:\Omega^{k}(M)\to\Omega^{k-1}(M)$,
and the Hodge Laplacian $\Delta=-\left(d\delta+\delta d\right)$ (cf. 
\cite[Section 2.10]{taylorPartialDifferentialEquations2011a} and 
\cite[Definition 1.2.2]{schwarzHodgeDecompositionMethod1995}).  We note that 
$d$ is the operator that appears in Stokes' theorem, and remark that $\delta$ is the 
$L^{2}$-adjoint of $d$; moreover, $\delta(X^{\flat})=-\Div_{g}X$ for 
$X\in\mathfrak{X}(M)$.

Throughout the paper, we will use the notation $D^{k}$ as a ``schematic'' for 
a spatial differential operator of order $k$, with coefficients bounded in
the $C^{\infty}$ topology for all time.  In each setting where this appears,
results from microlocal analysis (or just straight calculations) then give 
the schematic identities $$[D^{k},D^{l}]=D^{k+l-1},$$ when 
$\sigma(D^k)\circ \sigma(D^l)=\sigma(D^l)\circ \sigma(D^k)$ (e.g., 
when $D^k$ is a Laplace-type operator), and 
$$\left\langle \left\langle D^{k}\phi,\psi\right\rangle \right\rangle =\left\langle \left\langle \phi,D^{k}\psi\right\rangle \right\rangle,$$
for smooth tensor fields $\phi$ and $\psi$.

\subsection{Frequency cutoff and projection operators}
\label{sec:freq_cutoff}

As the Laplace-Beltrami operator $\Delta$ is self-adjoint, we can
define its functional calculus by the spectral 
theorem (see \cite[Section A.8]{taylorPartialDifferentialEquations2011a}).  For any 
$s\in\sigma\left(\sqrt{-\Delta}\right),$ define $\pi_{s}:\mathcal{D}'\left(M\right)\to C^{\infty}(M)$
as the continuous eigenspace projection mapping into the
space of eigenfunctions corresponding to $s$. So $\left(-\Delta\right)\pi_{s}=s^{2}\pi_{s}$
on $\mathcal{D}'\left(M\right)$. In particular, the image of $\pi_{0}$
is the space of constant functions. 

We also define the frequency cutoff projections $$P_{k}=1_{[k,k+1)}\left(\sqrt{-\Delta}\right)=\sum_{s\in\sigma\left(\sqrt{-\Delta}\right)\cap[k,k+1)}\pi_{s},\quad k>0,$$
and, for $k,m\in\mathbb{N}_{0}$, we define 
$$\mathcal{P}_{1}:H^{m}\Omega^{k}\left(M\right)\to d\left(H^{m+1}\Omega^{k-1}(M)\right),$$
$$\mathcal{P}_{2}:H^{m}\Omega^{k}\left(M\right)\to\delta\left(H^{m+1}\Omega^{k-1}(M)\right),$$
and
$$\mathcal{P}_{3}=\mathcal{P}_{\mathcal{H}}:\mathcal{D}'\Omega^{k}\left(M\right)\to\Omega^{k}(M),$$
as the continuous Hodge projections. As in \cite{huynhHodgetheoreticAnalysisManifolds2019},
we have $1=\mathcal{P}_{1}+\mathcal{P}_{2}+\mathcal{P}_{3}$, and remark that 
$$\mathbb{P}:=\mathcal{P}_{2}+\mathcal{P}_{3}$$ is the classical Leray
projection operator. 

Note that the range of $\mathcal{P}_{\mathcal{H}}$ is finite-dimensional
(with $\mathcal{P}_{\mathcal{H}}=\pi_{0}$ on $\Omega^{0}\left(M\right)$),
which is essentially the frequency zero. The foundational result of
Hodge theory is that for any $m\in\mathbb{N}_0$ and $k\in\mathbb{N}_{0}$,
$\Delta_{H}$ is bijective from $\left(1-\mathcal{P}_{\mathcal{H}}\right)H^{m+2}\Omega^{k}\left(M\right)$
to $\left(1-\mathcal{P}_{\mathcal{H}}\right)H^{m}\Omega^{k}\left(M\right)$.  It follows from this that 
$\left(-\Delta_{H}\right)^{-1}$ is well-defined on 
$\left(1-\mathcal{P}_{\mathcal{H}}\right)H^{m}\Omega^{k}\left(M\right)$.
We also easily see that $\mathcal{P}_{\mathcal{H}}\star=\star\mathcal{P}_{\mathcal{H}}$.  

We can extend $\mathcal{P}_{\mathcal{H}},\mathcal{P}_{1},\mathcal{P}_{2}$ to vector fields 
via the musical isomorphism on $1$-forms.  We can also define the frequency cutoffs $P_k$ on 
differential forms (and vector fields) by invoking the functional calculus of $\Delta_H$.

Then $P_kd=dP_k$, $P_k\delta=\delta P_k$, $P_k\Delta_H=\Delta_HP_k$, $P_k\star=\star P_k$, $P_k\mathcal{P}_i=\mathcal{P}_iP_k$ for
$i=1,2,3$.  In particular, $P_k\mathcal{P}_{\mathcal{H}}=0$ for any $k>0$.

Let $\lambda_1$ denote the smallest nonzero mode; that is,
\begin{align}
\lambda_1=\lambda_1(M):=\min (\sigma(\sqrt{-\Delta_H})\setminus \{0\}),\label{eq:lambda1}
\end{align}
where $\Delta_H$ is treated as an unbounded operator on $$L^2\Omega(M):=\oplus_{s=0}^2 L^2\Omega^s(M).$$

Then, for all $m\in\mathbb{N}_0$, we can replace the $H^m$ norm on forms by
\begin{align*}
\lVert \theta\rVert_{H^m}&=\lVert \mathcal{P}_{\mathcal{H}}\theta\rVert_{L^2}+\lVert k^m\lVert P_k\theta\rVert_{L^2}\rVert_{\ell_k^2(\mathbb{N}_0+\lambda_1)}\\
&\sim \lVert \mathcal{P}_{\mathcal{H}}\theta\rVert_{L^2}+\lVert (-\Delta_H)^{m/2}(1-\mathcal{P}_{\mathcal{H}})\theta\rVert_{L^2}.
\end{align*}

\subsection{The vector Laplacian operator}

We now make a few remarks related to the choice of the Laplacian operator $\Delta_M$ as either
the Hodge Laplacian $\Delta_H$, the connection (Bochner) Laplacian $\Delta_B$, or the deformation
Laplacian $\Delta_D$.  

The choice of this operator in the Navier-Stokes system (\ref{eq:NavierStokes})
affects the class of initial data that one can consider to obtain global results by perturbing 
the classical flat background theory.  For instance, if we choose the 
Hodge-Laplacian, the ``global existence for small data'' result on flat spaces will generalize 
to initial data near the space of harmonic vector fields (satisfying $\Delta_{H}X=0$).  
On the other hand, if we choose the deformation Laplacian, then,
as proved in \cite{pruessNavierStokesEquationsSurfaces2020}, the flat background
small data theory leads to results for initial data in small neighborhoods of
the space of {\it Killing vector fields}, which satisfy the 
equation $\Delta_{D}X=0$.  Both spaces are finite-dimensional 
and algebraically rigid, leading to their own respective theories.

Nevertheless, for the results of this paper, we may take $\Delta_M$ to be any of these three choices.  Indeed, our arguments 
will rely on the following properties, which are valid for all three operators:
\begin{enumerate}
\item[(i)] $\Delta_{M}=\Delta_{H}+F$ where $F$ is a differential operator of
order 0 (with smooth coefficients),
\item[(ii)] $\Delta_{M}$ is self-adjoint on $L^{2}\mathfrak{X}(M)$, and
\item[(iii)] $\Delta_M\leq 0$.
\end{enumerate}

Note that condition (iii) amounts to a choice of convention for signs, and
corresponds to the physical dissipation of energy.

\subsection{The vorticity equation}
\label{sec:vorticity}

Consider the Navier-Stokes equation as in (\ref{eq:NavierStokes}).  For $f\in C^\infty(M)$, define 
$\curl f=-\left(\star df\right)^{\sharp}$.  Because we are in two spatial dimensions, we have 
$dU^{\flat}=\omega\vol$ for some $\omega\in C^{\infty}\left(M\right)$.  The {\it vorticity} 
$\omega$ is then defined by setting $$\omega:=\star d\flat U=\star d\mathcal{P}_{2}\flat U.$$

The advantage of working in two spatial dimensions is that the vorticity can be identified
with a scalar function, and we can control its $L^{2}$ norm (via the enstrophy estimate).  It is 
trivial to check that $d:\mathcal{P}_{2}\Omega^{1}\left(M\right)\to\mathcal{P}_{1}\Omega^{2}\left(M\right)$
is bijective with inverse $R_{d}=\delta\left(-\Delta_{H}\right)^{-1}=\left(-\Delta_{H}\right)^{-1}\delta$,
so $\mathcal{P}_{2}\flat U=\delta\left(-\Delta_{H}\right)^{-1}\star\omega=-\star d\left(-\Delta\right)^{-1}\omega$.  Moreover, 
$\pi_{0}\omega=0$, and, since $U=\mathbb{P}U$, $$\left(1-\mathcal{P}_{\mathcal{H}}\right)U=\mathcal{P}_{2}U=\curl\left(-\Delta\right)^{-1}\omega.$$

We use these identities to reformulate the system (\ref{eq:NavierStokes}) in terms of $\omega$.  For this, we begin with a lemma relating the Lie
derivative and the musical isomorphism.
\begin{lem}
\label{lem:small_lie_lemma}
Let $X$ be a smooth vector field on $M$.  Then $$\mathcal{L}_{X}X^{\flat}=\nabla_{X}X^{\flat}+d\left(\frac{|X|^{2}}{2}\right).$$
\end{lem}

\begin{proof}
We compute, for any smooth vector field $Y\in\mathfrak{X}(M)$, 
\begin{align*}
\left(\mathcal{L}_{X}X^{\flat}\right)\left(Y\right) & =\mathcal{L}_{X}\left\langle X,Y\right\rangle -\left\langle X,[X,Y]\right\rangle \\
&=\left\langle \mathcal{\nabla}_{X}X,Y\right\rangle +\left\langle X,\nabla_{X}Y-[X,Y]\right\rangle \\
 & =\left\langle \mathcal{\nabla}_{X}X,Y\right\rangle +\left\langle X,\nabla_{Y}X\right\rangle \\
&=\left(\mathcal{\nabla}_{X}X^{\flat}\right)\cdot Y+d\left(\frac{|X|^{2}}{2}\right)\cdot Y.
\end{align*}
This completes the proof of the lemma.
\end{proof}

We apply the musical isomorphism to the Navier-Stokes equation.  In view of Lemma \ref{lem:small_lie_lemma}, this gives
\[
0=\partial_{t}U^{\flat}+\mathcal{L}_{U}U^{\flat}-\nu\left(\Delta_{H}+F\right)U^{\flat}+d\left(p-\frac{|U|^{2}}{2}\right).
\]
Now, an application of the exterior derivative leads to
\begin{align*}
0 & =\partial_{t}\omega\vol+\left(\mathcal{L}_{U}\omega\right)\vol+\omega(\Div U)\vol-\nu\left(\Delta\omega\vol+dF\flat U\right),
\end{align*}
which, using $\Div U=0$, becomes 
\begin{align*}
0 & =\partial_{t}\omega+\left\langle U,\nabla\omega\right\rangle -\nu\Delta\omega-\nu\star dF\flat U\\
 & =\partial_{t}\omega+\left\langle \curl\left(-\Delta\right)^{-1}\omega,\nabla\omega\right\rangle +\left\langle \mathcal{P}_{\mathcal{H}}U,\nabla\omega\right\rangle \\
&\hspace{0.2in}-\nu\Delta\omega-\nu\star dF\flat\left(\curl\left(-\Delta\right)^{-1}\omega+\mathcal{P}_{\mathcal{H}}U\right).
\end{align*}

Because $\omega$ only governs $\left(1-\mathcal{P}_{\mathcal{H}}\right)U$,
we cannot completely remove the coupling with the velocity equation.  Fortunately,
$\left\Vert U\right\Vert _{L^{2}}$ stays bounded (in view of the energy inequality),
so $\left\Vert \mathcal{P}_{\mathcal{H}}U\right\Vert _{C^{m}}\lesssim_m\left\Vert U\right\Vert _{L^{2}}$
stays bounded for all $m\in\mathbb{N}_{0}$.  This means that the harmonic part is relatively easy
to control.

Collecting these arguments, we have arrived at the following equivalent formulation of the Navier-Stokes system:
\begin{equation}
\left\{ \begin{array}{rl}
U & =\mathcal{P}_{\mathcal{H}}U+\curl\left(-\Delta\right)^{-1}\omega\\
0 & =\partial_{t}\omega+\left\langle \curl\left(-\Delta\right)^{-1}\omega,\nabla\omega\right\rangle +\left\langle \mathcal{P}_{\mathcal{H}}U,\nabla\omega\right\rangle \\
&\hspace{0.2in}+\nu D^{2}\left(-\Delta\right)^{-1}\omega+\nu D^{1}\mathcal{P}_{\mathcal{H}}U-\nu\Delta\omega\\
0 & =\partial_{t}\mathcal{P}_{\mathcal{H}}U+\mathcal{P}_{\mathcal{H}}\nabla_U U+\nu\mathcal{\mathcal{P}_{\mathcal{H}}}D^{0}U
\end{array}\right.\label{eq:vorticity}
\end{equation}
We note that the condition $\Div U=0$ is already implied, and operators
implied within $D^{2},D^{1},D^{0}$ can be explicitly written out.

\subsection{Galerkin approximation and a priori estimates}
\label{sec:galerkin}

Let $\lambda_{1}$ be the smallest nonzero mode as in (\ref{eq:lambda1}), and 
let $Z\subset\mathbb{N}_{0}+\lambda_{1}$ be a finite subset selecting the modes 
included in the Galerkin approximation $$\omega_{Z}=\sum\limits _{k\in Z}P_{k}\omega_{Z}.$$ 
Then the truncated equation is
\begin{equation}
\left\{ \begin{array}{rl}
U_{Z}= & \mathcal{P}_{\mathcal{H}}U_{Z}+\curl\left(-\Delta\right)^{-1}\sum\limits _{k\in Z}P_{k}\omega_{Z},\\
0= & \partial_{t}P_{k}\omega_{Z}+\sum\limits _{l_{1},l_{2}\in Z}P_{k}\left\langle \curl\left(-\Delta\right)^{-1}P_{l_{1}}\omega_{Z},\nabla P_{l_{2}}\omega_{Z}\right\rangle \\
&\hspace{0.2in}+\sum\limits _{l\in Z}P_{k}\left\langle \mathcal{P}_{\mathcal{H}}U_{Z},\nabla P_{l}\omega_{Z}\right\rangle+\sum\limits _{l\in Z}\nu P_{k}D^{2}\left(-\Delta\right)^{-1}P_{l}\omega_{Z}\\
&\hspace{0.2in}+\nu P_{k}D^{1}\mathcal{P}_{\mathcal{H}}U_{Z}-\nu\Delta P_{k}\omega_{Z},\;\;\;\;\forall k\in Z,\\
0= & \partial_{t}\mathcal{P}_{\mathcal{H}}U_{Z}+\mathcal{P}_{\mathcal{H}}\nabla_{U_{Z}}U_{Z}+\nu\mathcal{\mathcal{P}_{\mathcal{H}}}D^{0}U_{Z}.
\end{array}\right.\label{eq:vorticity-galerkin}
\end{equation}

A more explicit form, without $D^{2},D^{1},D^{0}$, is
\begin{equation}
\left\{ \begin{array}{rl}
U_{Z} & =\mathcal{P}_{\mathcal{H}}U_{Z}+\curl\left(-\Delta\right)^{-1}\omega_{Z},\\
0 & =\partial_{t}\omega_{Z}+P_{Z}\nabla_{U_{Z}}\omega_{Z}-\nu P_{Z}\star d\Delta_{M}\flat U_{Z},\\
0 & =\partial_{t}\mathcal{P}_{\mathcal{H}}U_{Z}+\mathcal{P}_{\mathcal{H}}\nabla_{U_{Z}}U_{Z}-\nu\mathcal{P}_{\mathcal{H}}\Delta_{M}U_{Z},
\end{array}\right.\label{eq:vorticity-galerkin-simple}
\end{equation}
where $P_{Z}:=\sum_{k\in Z}P_{k}$. 

Selecting a finite basis for $\Range\left(\mathcal{P}_{\mathcal{H}}\right)$
and $\Range\left(P_{k}\right)$ for each $k$, we obtain a smooth 
finite-dimensional ODE system (with the unknowns being an analogue of 
the sequence of Fourier coefficients, depending only on time).

This system has a smooth solution on $\left[0,T_{Z}\right)$ for some $T_{Z}\in(0,\infty]$ (by Picard's theorem).
Standard Hodge theory now shows that the solution $U_{Z}$ solves a truncated form of the Navier-Stokes equation.

\begin{lem}
Let $Z\subset\mathbb{N}_0+\lambda_1$ be a finite set.  Suppose that $U_Z$ solves (\ref{eq:vorticity-galerkin-simple}).  Then $U_Z$ is also a solution to the equation
\begin{equation}
\partial_{t}U_{Z}+\left(\mathcal{P}_{\mathcal{H}}+P_{Z}\mathcal{P}_{2}\right)\nabla_{U_{Z}}U_{Z}-\nu\left(\mathcal{P}_{\mathcal{H}}+P_{Z}\mathcal{P}_{2}\right)\Delta_{M}U_{Z}=0.\label{eq:galerkin_solves_NS}
\end{equation}
\end{lem}

\begin{proof}
We have $\omega_{Z}=\star d\flat U_{Z}$ and
\begin{align*}
-\partial_{t}U_{Z}&=-\partial_{t}\mathcal{P}_{\mathcal{H}}U_{Z}-\curl\left(-\Delta\right)^{-1}\partial_{t}\omega_{Z}\\
&=\mathcal{P}_{\mathcal{H}}\nabla_{U_{Z}}U_{Z}-\nu\mathcal{P}_{\mathcal{H}}\Delta_{M}U_{Z}\\
&\hspace{0.2in}+\curl\left(-\Delta\right)^{-1}\left(P_{Z}\nabla_{U_{Z}}\omega_{Z}-\nu P_{Z}\star d\Delta_{M}\flat U_{Z}\right).
\end{align*}

In view of this, we want to show 
\[
\mathcal{P}_{2}\left(P_{Z}\nabla_{U_{Z}}U_{Z}-\nu P_{Z}\Delta_{M}U_{Z}\right)=\curl\left(-\Delta\right)^{-1}\left(P_{Z}\nabla_{U_{Z}}\omega_{Z}-\nu P_{Z}\star d\Delta_{M}\flat U_{Z}\right)
\]
We showed above that $\curl\left(-\Delta\right)^{-1}:\mathcal{P}_{2}\Omega^{0}\to\mathcal{P}_{2}\mathfrak{X}$
is bijective with inverse $\star d\flat$. Also $d\mathcal{P}_{2}=d,$
so now we only need to show 
\[
\star d\flat\left(P_{Z}\nabla_{U_{Z}}U_{Z}-\nu P_{Z}\Delta_{M}U_{Z}\right)=P_{Z}\nabla_{U_{Z}}\omega_{Z}-\nu P_{Z}\star d\Delta_{M}\flat U_{Z}
\]

For this, note that by by \Lemref{small_lie_lemma}, we have
\begin{align*}
\star d\flat P_{Z}\nabla_{U_{Z}}U_{Z}&=P_{Z}\star d\mathcal{L}_{U_{Z}}\flat U_{Z}\\
&=P_{Z}\star\mathcal{L}_{U_{Z}}d\flat U_{Z}\\
&=P_{Z}\star\mathcal{L}_{U_{Z}}\left(\omega_{Z}\vol\right)\\
&=P_{Z}\star\left(\left(\mathcal{\nabla}_{U_{Z}}\omega_{Z}\right)\vol\right)\\
&=P_{Z}\mathcal{\nabla}_{U_{Z}}\omega_{Z},
\end{align*}
which completes the proof.
\end{proof}

We aim to take the limit $Z\uparrow\mathbb{N}_{0}$.  In order to obtain convergence, we will need a priori estimates that are independent of the truncation
set $Z$.  The first such estimate is the {\it energy inequality}, $$\left\Vert U_{Z}(t)\right\Vert _{L^{2}}\leq\left\Vert U_{Z}\left(0\right)\right\Vert _{L^{2}},$$
which follows from (\ref{eq:galerkin_solves_NS}) and the fact that $\Delta_{M}\leq0$.  In particular, the energy estimate implies 
that $\mathcal{P}_{\mathcal{H}}U_{Z}(t)$ stays bounded in the $C^{\infty}$ topology.

The second a priori estimate we will use is the {\it enstrophy estimate}, 
\begin{align}
\left\Vert \omega_{Z}\left(t\right)\right\Vert _{L^{2}}\lesssim_{\neg Z}\left(\left\Vert \omega_{Z}\left(0\right)\right\Vert _{L^{2}}+\left\Vert U_{Z}\left(0\right)\right\Vert _{L^{2}}\right)e^{\nu Ct},\label{eq:enstrophy}
\end{align}
which we will show holds for some $C$ depending only on $M$, not $Z$.  Indeed, observe that 
\begin{align*}
0 & =\left\langle \left\langle \partial_{t}\omega_{Z},\omega_{Z}\right\rangle \right\rangle +\left\langle \left\langle \nabla_{U_{Z}}\omega_{Z},\omega_{Z}\right\rangle \right\rangle -\nu\left\langle \left\langle \star d\left(\Delta_{H}+F\right)\flat U_{Z},\omega_{Z}\right\rangle \right\rangle \\
 & =\partial_{t}\left(\frac{\left\Vert \omega_{Z}\right\Vert _{2}^{2}}{2}\right)-\nu\left\langle \left\langle \star d\Delta_{H}\flat U_{Z},\star d\flat U_{Z}\right\rangle \right\rangle -\nu\left\langle \left\langle D^{1}U_{Z},\omega_{Z}\right\rangle \right\rangle \\
 & =\partial_{t}\left(\frac{\left\Vert \omega_{Z}\right\Vert _{2}^{2}}{2}\right)+\nu\left\langle \left\langle \Delta_{H}\flat U_{Z},\Delta_{H}\flat U_{Z}\right\rangle \right\rangle -\nu\left\langle \left\langle D^{1}U_{Z},\omega_{Z}\right\rangle \right\rangle,
\end{align*}
and thus $\partial_{t}\left(\left\Vert \omega_{Z}\right\Vert _{2}^{2}\right)\lesssim\nu\left\Vert \omega_{Z}\right\Vert _{2}\left\Vert U_{Z}\right\Vert _{H^{1}}$.
But, by the Poincare inequality \cite{huynhHodgetheoreticAnalysisManifolds2019},
$$\left\Vert U_{Z}\right\Vert _{H^{1}}\sim_{M}\left\Vert \mathcal{P}_{\mathcal{H}}U_{Z}\right\Vert _{2}+\left\Vert \delta\flat U_{Z}\right\Vert _{2}+\left\Vert d\flat U_{Z}\right\Vert _{2}\lesssim\left\Vert U_{Z}\left(0\right)\right\Vert _{2}+\left\Vert \omega_{Z}\right\Vert _{2}.$$

So we have 
\begin{align*}
\partial_{t}\left(\left\Vert \omega_{Z}\left(t\right)\right\Vert _{2}^{2}+\left\Vert U_{Z}\left(0\right)\right\Vert _{2}^{2}\right)&\lesssim\nu\left(\left\Vert \omega_{Z}\left(t\right)\right\Vert _{2}+\left\Vert U_{Z}\left(0\right)\right\Vert _{2}\right)^{2}\\
&\sim\nu\left(\left\Vert \omega_{Z}\left(t\right)\right\Vert _{2}^{2}+\left\Vert U_{Z}\left(0\right)\right\Vert _{2}^{2}\right).
\end{align*}
An application of Gronwall's inequality now gives the enstrophy estimate (\ref{eq:enstrophy}).

\begin{rem}
As a particular consequence, note that by Picard's theorem and the fact that modes are finite,
the above bounds show that $T_{Z}=\infty$ and that $U_{Z}$ exists globally in time.

Note that the enstrophy is non-increasing when $\Delta_{M}=\Delta_{H}$
($F=0$). This is the case for flat spaces.
\end{rem}

The main a priori estimate for smooth convergence we establish in this paper is the following theorem.
\begin{thm}
\label{thm:main_a_priori}If for some $A_{0}\in\left(0,\infty\right)$
and $r>1$, $$\left\Vert U_{Z}\left(0\right)\right\Vert _{2}\leq A_{0}\quad\textrm{and}\quad \left\Vert P_{k}\omega_{Z}\left(0\right)\right\Vert _{2}\leq\frac{A_{0}}{\left|k\right|^{r}}\;\forall k\in Z,$$
then 
\[
\left\Vert P_{k}\omega_{Z}\left(t\right)\right\Vert _{2}\leq\frac{A^{*}(t)}{\left|k\right|^{r}}\;\forall t\geq0,\forall k\in Z
\]
for some smooth $A^{*}(t)$ depending on $r,\nu,M,A_{0}$ and not
$Z$.
\end{thm}

Note that the hypotheses of Theorem \ref{thm:main_a_priori} implicitly yield 
$$\left\Vert w_{Z}\left(0\right)\right\Vert _{2}\lesssim A_{0}\left\Vert \frac{1}{k^{r}}\right\Vert _{l_{k}^{2}}\lesssim_{\neg Z} A_{0}.$$

We prove Theorem \ref{thm:main_a_priori} in Section \ref{sec:geom-trap} and Section \ref{sec:viscous} below; indeed, this
is the main task of the rest of this paper.  

Assume that we have established Theorem \ref{thm:main_a_priori}.  We now show how to conclude
the proof of our Theorem \ref{thm:general-manifold}, our main result on global regularity 
for solutions to the Navier-Stokes system (\ref{eq:NavierStokes}) on $M$.  To leverage 
the a priori bounds of Theorem \ref{thm:main_a_priori}, we begin with a short uniqueness lemma
for the class of smooth solutions.

\begin{lem}
\label{lem:uniqueness}
Any smooth solution to Navier-Stokes must be unique.
\end{lem}

\begin{proof}
Let $U$ and $U+V$ be 2 smooth solutions with the same initial data
(i.e. $V(0)=0$). Then $0=\partial_{t}V+\mathbb{P}\left(\nabla_{V}U+\nabla_{U+V}V\right)-\nu\Delta_{M}V$,
which implies
\[
\partial_{t}\left(\frac{\left\Vert V\right\Vert _{2}^{2}}{2}\right)=-\left\langle \left\langle \nabla_{V}U,V\right\rangle \right\rangle +\nu\left\langle \left\langle \Delta_{M}V,V\right\rangle \right\rangle \leq\left\Vert V\right\Vert _{2}^{2}\left\Vert \nabla U\right\Vert _{\infty}
\]
As $V(0)=0$, by Gronwall $V=0$.
\end{proof}

We now complete the proof of Theorem \ref{thm:general-manifold}.

\begin{proof}[Proof of Theorem \ref{thm:general-manifold}]

Suppose $U_0$ is smooth.  Then, choosing any $r>1$, we may apply Theorem 
\ref{thm:main_a_priori} to see that the sequence $(U_{Z})$ remains bounded 
in $C_{t,\mathrm{loc}}H_{x}^{\infty}$ as the truncation set
$Z$ expands to $\mathbb{N}_{0}+\lambda_{1}$.  

Using the usual exchange of one time derivative for spatial derivatives 
in our Navier-Stokes system, we therefore obtain uniform bounds in 
$C_{t,\mathrm{loc}}^{\infty}H_{x}^{\infty}$.  By the Sobolev embedding, 
it follows that there is a subsequence $(U_{Z_{i}})_{i\geq 0}$ converging
to a smooth solution $U$ (as $P_{Z}$ is a contraction on all $H^{m}$).

This shows that there exists a global smooth solution with initial data $U_0$.  In view of the
uniqueness result for solutions of this class given in Lemma \ref{lem:uniqueness}, the proof
of the global regularity result is complete.
\end{proof}

\section{Geometric trapping}
\label{sec:geom-trap}

In this section and Section \ref{sec:viscous}, we prove Theorem \ref{thm:main_a_priori}.  As described in the introduction, the
argument follows the general pattern of the geometric trapping method of \cite{mattinglyElementaryProofExistence1999}.  In our setting,
this requires several additional ideas, due to the less well-behaved spectral properties of the manifold $M$.

Assume the hypotheses in \Thmref{main_a_priori} are satisfied.  Fix $T>0$.  We want to show that there
is a positive constant $A_{T}^{*}>1$ (depending on $r$, $\nu$, $M$, $A_{0}$, and $T$,
but not on $Z$) such that
\begin{align}
\left\Vert P_{k}\omega_{Z}\left(t\right)\right\Vert _{2}\leq\frac{A_{T}^{*}}{\left|k\right|^{r}}\;\forall t\in\left[0,T\right],\forall k\in Z.\label{eq:trapping-goal}
\end{align}

By the enstrophy estimate, there is $\mathcal{E}_{T}^{*}>1$, which
may depend on $\nu$, $A_{0}$, $M$, and $T$, but does not depend on $Z\subset \mathbb{N}_0+\lambda_1$, such that 
$$\left\Vert \omega_{Z}\left(t\right)\right\Vert _{L^{2}}^{2}+\left\Vert U_{Z}\left(t\right)\right\Vert _{2}^{2}\leq\mathcal{E}_{T}^{*}\quad\textrm{for}\,\, t\in\left[0,T\right].$$
This means that for any $K_{0}>\lambda_1$, setting $$A_{1}:=\left(K_{0}^{r}+1\right)\left(\frac{A_{0}}{\sqrt{\mathcal{E}_{T}^{*}}}+1\right)+\lambda_{1},$$
we have 
\[
\left\Vert P_{k}\omega_{Z}\left(t\right)\right\Vert _{2}<\frac{A_{1}\sqrt{\mathcal{E}_{T}^{*}}}{\left|k\right|^{r}}\quad \textrm{for}\,\, t\in[0,T]\,\,\textrm{and}\,\, k\in\left(\mathbb{N}_{0}+\lambda_{1}\right)\cap\left[\lambda_{1},K_{0}\right],
\]
with $A_{0}<A_{1}\sqrt{\mathcal{E}_{T}^{*}}$.  This estimate handles the contribution of 
low frequencies. We also have $$P_{k}\omega_{Z}=0\;\forall k\in\left(\mathbb{N}_{0}+\lambda_{1}\right)\backslash Z.$$

Pick $K_{0}$ large (to be chosen later). We will show that $\left(\left\Vert P_{k}\omega_{Z}\left(t\right)\right\Vert _{2}\right)_{k\in\mathbb{N}_{0}+\lambda_{1}}$
remains trapped in the set
\[
\mathcal{\mathfrak{S}}\left(K_{0}\right)=\left\{ \left(a_{k}\right)_{k\in\mathbb{N}_{0}+\lambda_{1}}:a_{k}\leq\frac{A_{1}\sqrt{\mathcal{E}_{T}^{*}}}{\left|k\right|^{r}}\;\forall k\in\mathbb{N}_{0}+\lambda_{1}\right\} 
\]
Note that $\left(\left\Vert P_{k}\omega_{Z}\left(0\right)\right\Vert _{2}\right)_{k\in\mathbb{N}_{0}+\lambda_{1}}\in\mathcal{\mathfrak{S}}\left(K_{0}\right)$.
The idea is that if $$\left(\left\Vert P_{k}\omega_{Z}\left(t\right)\right\Vert _{2}\right)_{k\in\mathbb{N}_{0}+\lambda_{1}}$$
were to exit the set $\mathcal{\mathfrak{S}}\left(K_{0}\right)$ for some $t>0$, it would have to go through 
\begin{align*}
\mathcal{\mathfrak{S}}_{\partial}\left(K_{0}\right)=\bigg\{ \left(a_{k}\right)_{k\in\mathbb{N}_{0}+\lambda_{1}}:&a_{k}\leq\frac{A_{1}\sqrt{\mathcal{E}_{T}^{*}}}{\left|k\right|^{r}}\;\forall k\in\mathbb{N}_{0}+\lambda_{1},\\
&a_{\widetilde{k}}=\frac{A_{1}\sqrt{\mathcal{E}_{T}^{*}}}{\left|\widetilde{k}\right|^{r}}\text{ for some }\widetilde{k}\in\left(\mathbb{N}_{0}+\lambda_{1}\right)\cap\left(K_{0},\infty\right)\bigg\} 
\end{align*}

In other words, for some $k>K_{0}$, $\left\Vert P_{k}\omega_{Z}\left(t\right)\right\Vert _{2}$
must reach and then exceed $\frac{A_{1}\sqrt{\mathcal{E}_{T}^{*}}}{\left|k\right|^{r}}$.  
If we can show that when $\left\Vert P_{k}\omega_{Z}\left(t\right)\right\Vert _{2}=\frac{A_{1}\sqrt{\mathcal{E}_{T}^{*}}}{\left|k\right|^{r}}$,
we have $$\partial_{t}\left(\left\Vert P_{k}\omega_{Z}\left(t\right)\right\Vert _{2}^{2}\right)<0,$$
then the evolution cannot exit the confining set $\mathfrak{S}(K_0)$. 

To show this, we aim to show that the diffusion $\Delta P_{k}\omega_{Z}$ in (\ref{eq:vorticity-galerkin}) is
the dominant term.  Since 
$$\left\Vert \Delta P_{k}\omega_{Z}(t)\right\Vert _{2}\sim\frac{A_{1}\sqrt{\mathcal{E}_{T}^{*}}}{\left|k\right|^{r-2}},$$
our goal is to show bounds which yield a stronger decay rate than $\frac{1}{\left|k\right|^{r-2}}$.  In particular,
we reduce the proof of Theorem \ref{thm:main_a_priori} to the following lemma, which we will prove in Section \ref{sec:viscous}.

\begin{lem}[Viscous domination]
\label{lem:viscous_domination} Let $r$, $A_{0}$ and $\mathcal{E}_{T}^{*}$
be as above.  Let $K_{0}\geq\lambda_{1}+10$ be arbitrary and set 
$$A_{1}:=\left(K_{0}^{r}+1\right)\left(\frac{A_{0}}{\sqrt{\mathcal{E}_{T}^{*}}}+1\right)+\lambda_{1}\left(M\right).$$ 

Assume at time $t\in\left[0,T\right]$, we have $\left\Vert P_{l}\omega_{Z}\left(t\right)\right\Vert _{2}\leq\frac{A_{1}\sqrt{\mathcal{E}_{T}^{*}}}{\left|l\right|^{r}}\;\forall l\in\mathbb{N}_{0}+\lambda_{1}$.
Then for any $k\in\left(\mathbb{N}_{0}+\lambda_{1}\right)\cap\left(K_{0},\infty\right)$, we have
\begin{align*}
&\sum\limits _{l_{1},l_{2}\in Z}\left\Vert P_{k}\left\langle \curl\left(-\Delta\right)^{-1}P_{l_{1}}\omega_{Z}(t),\nabla P_{l_{2}}\omega_{Z}(t)\right\rangle \right\Vert _{2}\\
&\hspace{1.2in}+\sum\limits _{l\in Z}\left\Vert P_{k}\left\langle \mathcal{P}_{\mathcal{H}}U_{Z}(t),\nabla P_{l}\omega_{Z}(t)\right\rangle \right\Vert _{2}\\
&\hspace{1.2in}+\sum\limits _{l\in Z}\nu\left\Vert P_{k}D^{2}\left(-\Delta\right)^{-1}P_{l}\omega_{Z}(t)\right\Vert _{2}\\
&\hspace{1.2in}+\nu\left\Vert P_{k}D^{1}\mathcal{P}_{\mathcal{H}}U_{Z}(t)\right\Vert_{2}\lesssim_{\nu,M,r,\neg Z,\neg T,\neg K_{0}}\frac{A_{1}\mathcal{E}_{T}^{*}}{\left|k\right|^{r-\frac{7}{4}}}.
\end{align*}
\end{lem}

To conclude this section, we give the proof of Theorem \ref{thm:main_a_priori} under the assumption that we have already
shown Lemma \ref{lem:viscous_domination}.

\begin{proof}[Proof of Theorem \ref{thm:main_a_priori}]

Fit $T>0$, choose $K_0$ large, and let $A_1$ be as defined above.  Then if $0<t<T$ is such that $$\left\Vert P_{k}\omega_{Z}\left(t\right)\right\Vert _{2}=\frac{A_{1}\sqrt{\mathcal{E}_{T}^{*}}}{\left|k\right|^{r}}$$
and $$\left\Vert P_{l}\omega_{Z}\left(t\right)\right\Vert _{2}\leq\frac{A_{1}\sqrt{\mathcal{E}_{T}^{*}}}{\left|l\right|^{r}}\quad\textrm{for all}\,\, l\in\mathbb{N}_{0}+\lambda_{1},$$
it follows that we have
\begin{align*}
\partial_{t}\left(\frac{\left\Vert P_{k}\omega_{Z}\left(t\right)\right\Vert _{2}^{2}}{2}\right)&=O_{\nu,M,r,\neg Z,\neg T,\neg K_{0}}\left(\frac{A_{1}\mathcal{E}_{T}^{*}}{\left|k\right|^{r-\frac{7}{4}}}\left\Vert P_{k}\omega_{Z}\left(t\right)\right\Vert _{2}\right)\\
&\hspace{1.2in}+\nu\left\langle \left\langle \Delta P_{k}\omega_{Z}(t),P_{k}\omega_{Z}\left(t\right)\right\rangle \right\rangle.
\end{align*}

Observe that $\left\langle \left\langle \Delta P_{k}\omega_{Z}(t),P_{k}\omega_{Z}\left(t\right)\right\rangle \right\rangle <0$
and $$\left|\left\langle \left\langle \Delta P_{k}\omega_{Z}(t),P_{k}\omega_{Z}\left(t\right)\right\rangle \right\rangle \right|\sim_{M,\neg Z,\neg T,\neg K_{0}}\frac{A_{1}\sqrt{\mathcal{E}_{T}^{*}}}{\left|k\right|^{r-2}}\left\Vert P_{k}\omega_{Z}\left(t\right)\right\Vert _{2}.$$
In particular, we can choose $K_{0}$ so that $\frac{\sqrt{\mathcal{E}_{T}^{*}}}{K_{0}^{1/4}}\ll_{\nu,M,r,\neg Z,\neg T,\neg K_{0}}1$, thereby obtaining $$\partial_{t}\left(\frac{\left\Vert P_{k}\omega_{Z}\left(t\right)\right\Vert _{2}^{2}}{2}\right)<0.$$

Then $\left(\left\Vert P_{k}\omega_{Z}\left(t\right)\right\Vert _{2}\right)\in\mathcal{\mathfrak{S}}\left(K_{0}\right)$ for all $t\in[0,T]$,
and the desired bound (\ref{eq:trapping-goal}) follows by setting $A_{T}^{*}=A_{1}\sqrt{\mathcal{E}_{T}^{*}}$.
\end{proof}

\section{Viscous domination}
\label{sec:viscous}

In this section, we prove \Lemref{viscous_domination}.  To frame our techniques, we recall a clasical result used in the study of the
nonlinear Schr\"odinger equation.

\begin{proposition}
\label{proposition:bilinear_1_old} For any $f,g\in L^{2}\left(M\right)$
and $l_{1},l_{2}\geq\lambda_{1}(M)$ and $a,b\in\mathbb{N}_{0}$,
we have
\[
\left\Vert \left(\nabla^{a}P_{l_{1}}f\right)*\left(\nabla^{b}P_{l_{2}}f\right)\right\Vert _{2}\lesssim_{M}\min\left(l_{1},l_{2}\right)^{\frac{1}{4}}l_{1}^{a}\left\Vert P_{l_{1}}f\right\Vert _{2}l_{2}^{b}\left\Vert P_{l_{2}}f\right\Vert _{2}
\]
where $\left(\nabla^{a}P_{l_{1}}f\right)*\left(\nabla^{b}P_{l_{2}}f\right)$
is schematic for any contraction of the two tensors.
\end{proposition}

\begin{proof}
Let $\chi\in \mathcal{S}(\mathbb{R})$ such that $\chi=1$ on $[0,1]$. Define $\chi_\lambda=\chi(\sqrt{-\Delta}-\lambda)$. Then $\chi_{l_1}P_{l_1}f=P_{l_1}f$, and we can use \cite[Equation 7.13]{haniGlobalWellposednessCubic2011}.\footnote{The statement 
in \cite{haniGlobalWellposednessCubic2011} contains a slight typo \cite{haniPersonalCommunication}.
The correct factor $\min\left(l_{1},l_{2}\right)^{\frac{1}{4}}$ was originally given in
\cite{burqMultilinearEigenfunctionEstimates2005}.}
\end{proof}

We will make use of a variant of this result, adapted to the frequency cutoff operators
defined in Section \ref{sec:freq_cutoff}.

Since we want to avoid relying on facts about the distribution of eigenvalues, we will
use a Fourier decomposition technique, decomposing the multilinear symbols into linear 
pieces (see, e.g. \cite[Lemma 2.10]{christRegularityInversesSingular1988}
or \cite[Proposition 7.5]{haniGlobalWellposednessCubic2011}).  This strategy, which we 
will refer to in the rest of this paper as the ``Fourier trick,'' allows us to pass 
between frequency cutoffs $P_k$ and eigenspace projections $\pi_s$.

\begin{lem}[Bilinear estimate]
\label{lem:bilinear_1_new} For any $f,g\in L^{2}\left(M\right)$
and $l_{1},l_{2}\geq\lambda_{1}\left(M\right)$ and $a,b,c\in\mathbb{N}_{0}$,
we have
\[
\left\Vert \left(\nabla^{a}P_{l_{1}}f\right)*\left(\nabla^{b}\left(-\Delta\right)^{-c}P_{l_{2}}g\right)\right\Vert _{2}\lesssim_{M}\min\left(l_{1},l_{2}\right)^{\frac{1}{4}}l_{1}^{a}\left\Vert P_{l_{1}}f\right\Vert _{2}l_{2}^{b-2c}\left\Vert P_{l_{2}}g\right\Vert _{2}
\]
where $\left(\nabla^{a}P_{l_{1}}f\right)*\left(\nabla^{b}P_{l_{2}}g\right)$
is schematic for any contraction of the two tensors.
\end{lem}

The main intuition underlying the connection between Proposition \ref{proposition:bilinear_1_old} 
and Lemma \ref{lem:bilinear_1_new} is that $\left(-\Delta\right)^{-c}P_{l_{2}}$
is almost like $l_{2}^{-2c}P_{l_{2}}$ (but not quite, as frequency cutoffs are a bit 
different from eigenspace projections). 

\begin{proof}[Proof of Lemma \ref{lem:bilinear_1_new}]
Let $h\in L^{2}\left(M\right)$ such that $\left\Vert h\right\Vert _{2}\leq1.$
We want to show
\begin{align}
\nonumber &\left\langle \left\langle \left(\nabla^{a}P_{l_{1}}f\right)*\left(\nabla^{b}\left(-\Delta\right)^{-c}P_{l_{2}}g\right),h\right\rangle \right\rangle \\
&\hspace{0.2in}=O\left(\min\left(l_{1},l_{2}\right)^{\frac{1}{4}}l_{1}^{a}\left\Vert P_{l_{1}}f\right\Vert _{2}l_{2}^{b-2c}\left\Vert P_{l_{2}}g\right\Vert _{2}\right).\label{eq:bilinear_eqn}
\end{align}

Observe that, using eigenspace projections, the left-hand side of (\ref{eq:bilinear_eqn}) can be rewritten as
\[
\sum\limits _{\substack{z_{j}\in[0,1)\cap\left(\sigma\left(\sqrt{-\Delta}\right)-l_{j}\right)\\
j=1,2
}
}\frac{1}{\left(l_{2}+z_{2}\right)^{2c}}\left\langle \left\langle \left(\nabla^{a}\pi_{l_{1}+z_{1}}f\right)*\left(\nabla^{b}\pi_{l_{2}+z_{2}}g\right),h\right\rangle \right\rangle 
\]
Set $\psi\left(z_{1},z_{2}\right)=\frac{1}{\left(l_{2}+z_{2}\right)^{2c}}$
for $z_{1},z_{2}\in[0,1]$. Simple calculations show that $$\left\Vert \psi\right\Vert _{C^{2}\left([0,1]^{2}\right)}\lesssim_{\lambda_{1}}\frac{1}{l_{2}^{2c}}.$$

We can easily extend $\psi$ to a $C^{2}$ function $\psi_{1}$
supported in $\left(-2,2\right)^{2}$ such that $\left\Vert \psi_{1}\right\Vert _{C^{2}}\lesssim\left\Vert \psi\right\Vert _{C^{2}}$.
Then we can treat $\psi_{1}$ as a $C^{2}$ periodic function on $[-2,2]^{2}$
and apply Fourier inversion: $\psi_{1}\left(z_{1},z_{2}\right)=\sum_{\theta_{1},\theta_{2}\in\frac{\mathbb{Z}}{4}}\widehat{\psi_{1}}\left(\theta_{1},\theta_{2}\right)e^{i2\pi z\cdot\theta}$
with 
\[
\left\Vert \widehat{\psi_{1}}\right\Vert _{l^{1}}\lesssim\left\Vert \left\langle \theta\right\rangle ^{2}\widehat{\psi_{1}}(\theta)\right\Vert _{l_{\theta}^{2}\left(\frac{\mathbb{Z}^{2}}{4}\right)}\left\Vert \frac{1}{\left\langle \theta\right\rangle ^{2}}\right\Vert _{l_{\theta}^{2}\left(\frac{\mathbb{Z}^{2}}{4}\right)}\lesssim\left\Vert \psi_{1}\right\Vert _{H^{2}}\lesssim\left\Vert \psi\right\Vert _{C^{2}}\lesssim\frac{1}{l_{2}^{2c}}.
\]

We can then rewrite the left-hand side of (\ref{eq:bilinear_eqn}) as 
\begin{align*}
&\sum\limits _{z_1,z_2}
\sum_{\theta_{1},\theta_{2}\in\frac{\mathbb{Z}}{4}}\widehat{\psi_{1}}\left(\theta_{1},\theta_{2}\right)\left\langle \left\langle \left(\nabla^{a}\pi_{l_{1}+z_{1}}fe^{i2\pi z_{1}\theta_{1}}\right)*\left(\nabla^{b}\pi_{l_{2}+z_{2}}ge^{i2\pi z_{2}\theta_{2}}\right),h\right\rangle \right\rangle \\
&\hspace{1.7in}=\sum_{\theta_{1},\theta_{2}\in\frac{\mathbb{Z}}{4}}\widehat{\psi_{1}}\left(\theta_{1},\theta_{2}\right)\left\langle \left\langle \left(\nabla^{a}P_{l_{1}}f_{\theta_{1}}\right)*\left(\nabla^{b}P_{l_{2}}g_{\theta_{2}}\right),h\right\rangle \right\rangle,
\end{align*}
where the outer sum on the left hand side is over the index set $$\{(z_1,z_2):z_{j}\in[0,1)\cap\left(\sigma\left(\sqrt{-\Delta}\right)-l_{j}\right)\,\,\textrm{for}\,\, j=1,2\},$$ and 
where we've set $$f_{\theta_{1}}:=\sum\limits _{\substack{z_{1}\in[0,1)\cap\left(\sigma\left(\sqrt{-\Delta}\right)-l_{1}\right)}
}\pi_{l_{1}+z_{1}}fe^{i2\pi z_{1}\theta_{1}}$$ and $$g_{\theta_{2}}:=\sum\limits _{\substack{z_{2}\in[0,1)\cap\left(\sigma\left(\sqrt{-\Delta}\right)-l_{2}\right)}
}\pi_{l_{2}+z_{2}}ge^{i2\pi z_{2}\theta_{2}}.$$ 

The crucial point is that the $L^{2}$ norm is modulation-independent,
and the eigenspaces are mutually orthogonal, so $\left\Vert P_{l_{1}}f_{\theta_{1}}\right\Vert _{2}=\left\Vert P_{l_{1}}f\right\Vert _{2}$,
$\left\Vert P_{l_{2}}g_{\theta_{2}}\right\Vert _{2}=\left\Vert P_{l_{2}}g\right\Vert _{2}$.

Applying Proposition \Blackboxref{bilinear_1_old} and noting that 
\begin{align*}
&\sum_{\theta_{1},\theta_{2}\in\frac{\mathbb{Z}}{4}}\widehat{\psi_{1}}\left(\theta_{1},\theta_{2}\right)O_{\neg \theta_1,\neg\theta_2}\left(\min\left(l_{1},l_{2}\right)^{\frac{1}{4}}l_{1}^{a}\left\Vert P_{l_{1}}f\right\Vert _{2}l_{2}^{b}\left\Vert P_{l_{2}}g\right\Vert _{2}\right)\\
&\hspace{0.2in}=O\left(\left\Vert \widehat{\psi_{1}}\right\Vert _{l^{1}}\min\left(l_{1},l_{2}\right)^{\frac{1}{4}}l_{1}^{a}\left\Vert P_{l_{1}}f\right\Vert _{2}l_{2}^{b}\left\Vert P_{l_{2}}g\right\Vert _{2}\right),
\end{align*}
the desired conclusion follows.
\end{proof}

While the estimate established in Lemma \ref{lem:bilinear_1_new} is good enough for the so-called ``triangle
regions'' in our analysis (see Claim \ref{clm:clm1} in Section \ref{sec:viscous}), we will also need another estimate
to treat the ``distant regions'' of frequency interactions.  In \cite{haniGlobalWellposednessCubic2011}, Hani
showed that for any $f,g,h\in L^{2}\left(M\right)$ and $l_{1}\geq l_{2}\geq l_{3}\geq\lambda_{1}(M)$
such that $l_{1}=l_{2}+Kl_{3}+2$ for $K>1$ , one has
\[
\left|\int_{M}\left(P_{l_{1}}f\right)\left(P_{l_{2}}g\right)\left(P_{l_{3}}h\right)\right|\lesssim_{J,M}\frac{l_{3}^{\frac{1}{4}}}{K^{J}}\left\Vert P_{l_{1}}f\right\Vert _{2}\left\Vert P_{l_{2}}g\right\Vert _{2}\left\Vert P_{l_{3}}h\right\Vert _{2}
\]
for all $J\in\mathbb{N}_{0}$.

We generalize this result to the following lemma.
\begin{lem}[Trilinear estimate]
\label{lem:trilinear_new}For any $f_{1},f_{2},f_{3}\in L^{2}\left(M\right)$;
$a_{1},b_{1},a_{2},b_{2},a_{3},b_{3},J\in\mathbb{N}_{0}$ and $l_{1}\geq l_{2}\geq l_{3}\geq\lambda_{1}(M)$
such that $l_{1}=l_{2}+Kl_{3}+2$ for $K>1$, we have 
\begin{align*}
&\left|\int_{M}\left(\nabla^{a_{1}}\left(-\Delta\right)^{-b_{1}}P_{l_{1}}f_{1}\right)*\left(\nabla^{a_{2}}\left(-\Delta\right)^{-b_{2}}P_{l_{2}}f_{2}\right)*\left(\nabla^{a_{3}}\left(-\Delta\right)^{-b_{3}}P_{l_{3}}f_{3}\right)\right|\\
&\hspace{0.2in}\lesssim_{J,M,\neg l_{1},\neg l_{2},\neg l_{3}}\frac{l_{3}^{\frac{1}{4}}}{K^{J}}\prod_{j=1}^{3}l_{j}^{a_{j}-2b_{j}}\left\Vert P_{l_{j}}f_{j}\right\Vert _{2}
\end{align*}
\end{lem}

The proof of Lemma \ref{lem:trilinear_new} is given in Appendix \ref{appendix:lemma}.  The ideas
involved are similar to the tools used in the proof of Lemma \ref{lem:bilinear_1_new} above.

We now proceed to the proof of Lemma \ref{lem:viscous_domination}.  To make the argument easier to follow, we note
that it suffices to establish the following self-contained statement.  This formulation makes it clear that there 
is no dependence on $K_{0},T,Z$ in \Lemref{viscous_domination}. 

\begin{nlem}{$\textrm{7}^{\,\prime}$}[Viscous domination, restated]
\label{lem:viscous_domination_restated}
Let $w\in C^{\infty}(M)$ and $u\in\mathcal{P}_{\mathcal{H}}\mathfrak{X}\left(M\right)$.  Let $A,B\geq1$ and $k\in\mathbb{N}_{0}+\lambda_{1}+10$. Let $r>1$.  Assume that $\pi_{0}w=0$ and $\left\Vert P_{l}w\right\Vert _{2}\leq\frac{A}{\left|l\right|^{r}}$ for all $l\in\mathbb{N}_{0}+\lambda_{1}$.  Assume also that $\left\Vert w\right\Vert _{2}+\left\Vert u\right\Vert _{2}=\left\Vert \left\Vert P_{j}w\right\Vert _{2}\right\Vert _{l_{j}^{2}\left(\mathbb{N}_{0}+\lambda_{1}\right)}+\left\Vert u\right\Vert _{2}\leq B$.

Then 
\begin{align*}
&\sum\limits _{l_{1},l_{2}\in\mathbb{N}_{0}+\lambda_{1}}\left\Vert P_{k}\left\langle \curl\left(-\Delta\right)^{-1}P_{l_{1}}w,\nabla P_{l_{2}}w\right\rangle \right\Vert _{2}+\sum\limits _{l\in\mathbb{N}_{0}+\lambda_{1}}\left\Vert P_{k}\left\langle \mathcal{P}_{\mathcal{H}}u,\nabla P_{l}w\right\rangle \right\Vert _{2}\\
&\hspace{0.8in}+\sum\limits _{l\in\mathbb{N}_{0}+\lambda_{1}}\left\Vert P_{k}D^{2}\curl\left(-\Delta\right)^{-1}P_{l}w\right\Vert _{2}+\left\Vert P_{k}D^{1}\mathcal{P}_{\mathcal{H}}u\right\Vert _{2} \lesssim_{M,r}\frac{AB}{\left|k\right|^{r-\frac{7}{4}}}
\end{align*}
\end{nlem}

We will split this problem into smaller claims, handling the contribution of each term.  This is the content of the next three
subsections, the combination of which together establish Lemma \ref{lem:viscous_domination_restated}.

\subsection{The convective term}

The main tools we will use are \Lemref{bilinear_1_new} and \Lemref{trilinear_new}.
We also note that for $p\in\left[1,\infty\right]$ and $\alpha\in\mathbb{R}$:
\begin{itemize}
\item $\left\Vert l^{\alpha}\right\Vert _{l_{1\lesssim l\lesssim k}^{p}}\lesssim k^{\alpha+\frac{1}{p}}$
for $\alpha p>-1$.
\item $\left\Vert \frac{1}{l^{\alpha}}\right\Vert _{l_{l\gtrsim k}^{p}}\sim\frac{1}{k^{\alpha-\frac{1}{p}}}$
for $\alpha p>1$.
\end{itemize}
\begin{claim}
\label{clm:clm1}
Firstly, we will show 
\[
\sum\limits _{\substack{l_{1},l_{2}\in\mathbb{N}_{0}+\lambda_{1}\\
\left|l_{1}-l_{2}\right|\leq k\leq l_{1}+l_{2}
}
}\left\Vert P_{k}\left\langle \curl\left(-\Delta\right)^{-1}P_{l_{1}}w,\nabla P_{l_{2}}w\right\rangle \right\Vert _{2}\lesssim\frac{AB}{k^{r-\frac{7}{4}}}
\]
\end{claim}

\begin{rem*}
These ``triangle regions'' are all we need to complete the proof of 
Lemma \ref{lem:viscous_domination_restated} (and its original formulation Lemma \ref{lem:viscous_domination}), and thus also of 
Theorem \ref{thm:general-manifold}, in the case when $M$ is the 
sphere $S^2\subset\mathbb{R}^3$. Indeed, on the sphere, we have $\Ric(X)=X$ and $\mathcal{P}_{\mathcal{H}}=0$.
In this case, we are therefore justified in setting $\Delta_{M}=\Delta_{H}+c$ where $c\in\mathbb{R}$
is a constant, which is easy to handle as $c\left\Vert P_{k}\omega_{Z}\right\Vert _{2}\lesssim_{\lambda_{1}}\left\Vert \Delta P_{k}\omega_{Z}\right\Vert _{2}$. Also, we have $P_{k}\left\langle \curl\left(-\Delta\right)^{-1}P_{l_{1}}\omega_{Z}(t),\nabla P_{l_{2}}\omega_{Z}(t)\right\rangle =0$
if $\left(k,l_{1},l_{2}\right)$ does not obey the triangle inequalities; see
\cite[Equation (26)]{caoNavierStokesEquationsRotating1999}.
\end{rem*}

\begin{proof}[Proof of Claim \ref{clm:clm1}]
Let $\mathcal{T}=\left\{ \left(l_{1},l_{2}\right):l_{1},l_{2}\in\mathbb{N}_{0}+\lambda_{1}\text{ and }\left|l_{1}-l_{2}\right|\leq k\leq l_{1}+l_{2}\right\} $.
We write $\mathcal{T}=\mathcal{T}_{1}\cup\mathcal{T}_{2}\cup\mathcal{T}_{3}$, where the sets $\mathcal{T}_i$ are defined by
$$\mathcal{T}_{1}=\{\left(l_{1},l_{2}\right)\in\mathcal{T}:l_{1}\leq\frac{k}{2}\},$$
$$\mathcal{T}_{2}=\{\left(l_{1},l_{2}\right)\in\mathcal{T}:\frac{k}{2}<l_{1}\leq2k\},$$
and
$$\mathcal{T}_{3}=\{\left(l_{1},l_{2}\right)\in\mathcal{T}:l_{1}>2k\}.$$

\begin{figure}
	\centering
	\includegraphics[width=0.7\linewidth]{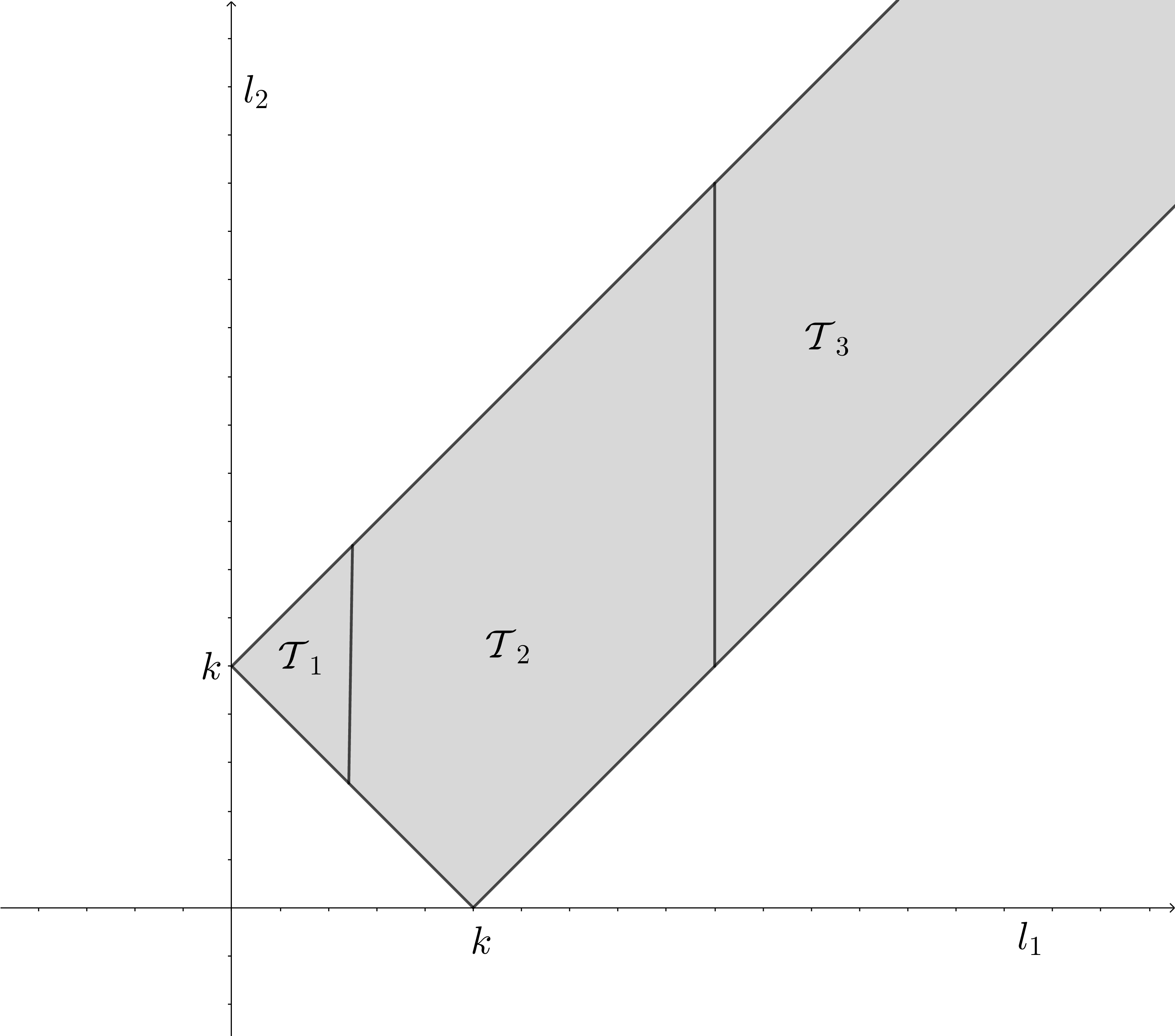}
	\caption{The triangle regions.}
	\label{fig:navierdiagram1}
\end{figure}

We begin by estimating the contribution of $\mathcal{T}_1$.  In this case, we have $l_{1}\leq l_{2}\sim k$, and the contribution of terms from $\mathcal{T}_1$ is bounded by
\begin{align}
\sum_{l_{1}}\sum_{l_{2}}l_{1}^{1/4}\frac{1}{l_{1}}\left\Vert P_{l_{1}}w\right\Vert _{2}k\left\Vert P_{l_{2}}w\right\Vert _{2}\lesssim\sum_{l_{1}}l_{1}^{1/4}\left\Vert P_{l_{1}}w\right\Vert _{2}k\frac{A}{k^{r}}\label{eq:t1}
\end{align}
where to obtain the last inequality we have noted that for each $l_{1}$, there are at most $2l_{1}$ choices of $l_{2}$.  The H\"older inequality now gives the bound
\begin{align*}
(\ref{eq:t1})\lesssim k^{3/4}B\cdot k\frac{A}{k^{r}}=\frac{AB}{k^{r-7/4}}
\end{align*}

We now estimate the contribution from $\mathcal{T}_2$.  Here we have $l_{2}\lesssim k\sim l_{1}$.  The contribution is then bounded by 
\begin{align}
\nonumber \sum_{l_{1}}\sum_{l_{2}}k^{1/4}\frac{1}{k}\left\Vert P_{l_{1}}w\right\Vert _{2}l_{2}\left\Vert P_{l_{2}}w\right\Vert _{2}&\lesssim
\sum_{l_{1}}k^{1/4}\frac{1}{k}\left\Vert P_{l_{1}}w\right\Vert _{2}k^{\frac{3}{2}}B\\
&\leq\sum_{l_{1}}\frac{AB}{k^{r-3/4}}\label{eq:t2}
\end{align}
where the H\"older inequality is used in passing from left to right in the first line.  Recalling that there are at most $O(k)$ choices for
the value of $l_1$ in the summation for this contribution, we obtain the bound
\begin{align*}
(\ref{eq:t2})\lesssim\frac{AB}{k^{r-7/4}}.
\end{align*}

It remains to estimate the $\mathcal{T}_3$ contribution.  For this, we have $k\lesssim l_{1}\sim l_{2}$, and we note that, for each fixed $l_1$, the 
number of choices for $l_2$ is at most $O(k)$.  Making the change of variable $l_{2}=l_{1}+j,$ where
$\left|j\right|\leq k$, the contribution of $\mathcal{T}_3$ is bounded by
\begin{align*}
\sum_{j}\sum_{l_{1}}l_{1}^{1/4}\frac{1}{l_{1}}\left\Vert P_{l_{1}}w\right\Vert _{2}\left(l_{1}+j\right)\left\Vert P_{l_{1}+j}w\right\Vert _{2}&\lesssim\sum_{j}\sum_{l_{1}}\frac{A}{l_{1}^{r-1/4}}\left\Vert P_{l_{1}+j}w\right\Vert _{2}\\
&\lesssim \sum_{j}\frac{A}{k^{r-3/4}}B\\
&\lesssim \frac{AB}{k^{r-7/4}},
\end{align*}
where we used the H\"older inequality to pass from the first to second lines.  Note that in this calculation we needed $2\left(r-\frac{1}{4}\right)>1$.
\end{proof}

As we oberved above, this completes the proof in the case of the sphere $M=S^2$.  To treat more 
general manifolds, we will invoke the trilinear estimate in \Lemref{trilinear_new} to estimate the 
contribution of the ``distant regions'' (where $\max(k,l_{1},l_{2})$ is far bigger than the rest).  
In addition, between the triangle regions  and the distant regions, there are ``intermediate 
regions'' where we require more ad-hoc arguments.

\begin{claim}
With $k\in\mathbb{N}_{0}+\lambda_{1}+10$, set $$\mathcal{A}:=\left\{ \left(l_{1},l_{2}\right):l_{1},l_{2}\in\mathbb{N}_{0}+\lambda_{1}\text{ and }\left|l_{1}-l_{2}\right|>k\right\},$$
and $$\mathcal{B}=\left\{ \left(l_{1},l_{2}\right):l_{1},l_{2}\in\mathbb{N}_{0}+\lambda_{1}\text{ and }l_{1}+l_{2}<k\right\}.$$

Then
\[
\sum\limits _{\left(l_{1},l_{2}\right)\in\mathcal{A}\cup\mathcal{B}}\left\Vert P_{k}\left\langle \curl\left(-\Delta\right)^{-1}P_{l_{1}}w,\nabla P_{l_{2}}w\right\rangle \right\Vert _{2}\lesssim_{M}\frac{AB}{k^{r-\frac{7}{4}}}
\]
\end{claim}

\begin{figure}
	\centering
	\includegraphics[width=\linewidth]{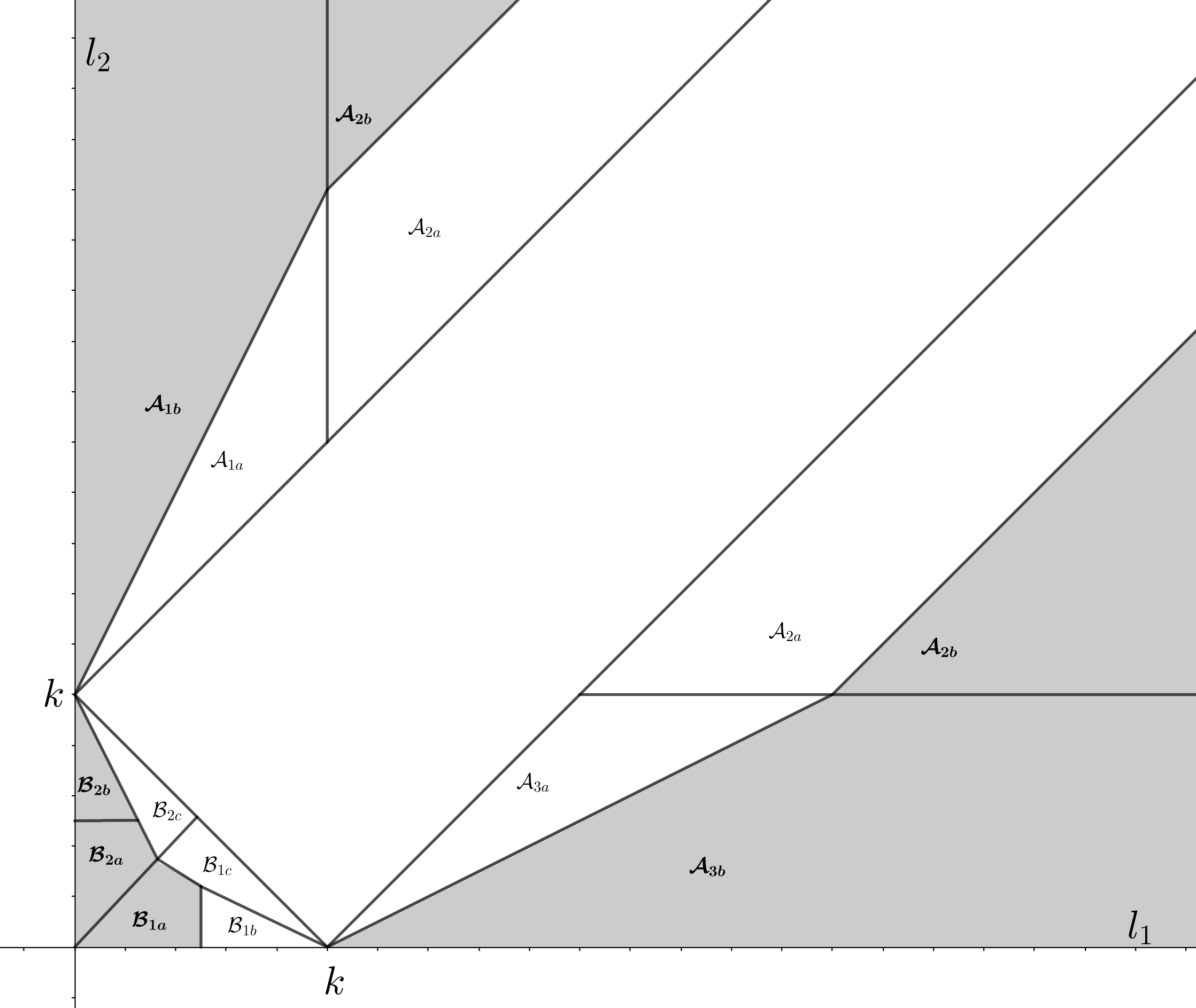}
	\caption{The non-triangle regions. Shaded regions indicate where the trilinear estimate of Lemma \ref{lem:trilinear_new} is used.}
	\label{fig:navierdiagram2}
\end{figure}

\begin{proof}
In using \Lemref{trilinear_new}, we will set $J$ as large as necessary.

We split $\mathcal{A}=\{\left|l_{1}-l_{2}\right|>k\}$ into smaller
regions
$$\mathcal{A}_1:=\{l_1\leq k\} \cap \mathcal{A},$$
$$\mathcal{A}_2:=\{l_1\geq k,l_2\geq k\} \cap \mathcal{A},$$
and
$$\mathcal{A}_3:=\{ l_{1}\geq k>l_{2}\} \cap \mathcal{A}.$$

We begin by estimating the contribution of $\mathcal{A}_1$.  For this, we consider the
contribution
$$\mathcal{A}_{1a}=\{l_{1}\leq k\leq l_{2}\leq k+2l_{1}+2\} \cap \mathcal{A},$$
for which $l_{1}\leq l_{2}\sim k$, and for each fixed $l_1$, there are at most $O(l_1)$ choices for the index $l_2$.  The contribution is then bounded by
\begin{align*}
&\sum_{l_{1}}\sum_{l_{2}}l_{1}^{1/4}\frac{1}{l_{1}}\left\Vert P_{l_{1}}w\right\Vert _{2}k\left\Vert P_{l_{2}}w\right\Vert _{2}\\
&\hspace{0.2in}\lesssim \sum_{l_{1}}l_{1}^{1/4}\left\Vert P_{l_{1}}w\right\Vert _{2}k\frac{A}{k^{r}}\\
&\hspace{0.2in}\lesssim k^{3/4}B\cdot\frac{A}{k^{r-1}}.
\end{align*}
To handle the contribution
$$\mathcal{A}_{1b}=\{l_{1}\leq k<k+2l_{1}+2<l_{2}\} \cap \mathcal{A},$$ where $k+2l_{1}+2\sim k$, we invoke
\Lemref{trilinear_new}, to bound the contribution by 
\begin{align}
\nonumber &\sum_{l_{1}}\sum_{l_{2}}l_{1}^{1/4}\frac{l_{1}^{J}}{\left(l_{2}-k-2\right)^{J}}\frac{1}{l_{1}}\left\Vert P_{l_{1}}w\right\Vert _{2}l_{2}\left\Vert P_{l_{2}}w\right\Vert _{2}\\
&\hspace{0.2in}\leq\sum_{l_{1}}l_{1}^{J-3/4}\left\Vert P_{l_{1}}w\right\Vert _{2}\sum_{l_{2}}\frac{A}{l_{2}^{r-1}\left(l_{2}-k-2\right)^{J}}\label{eq:a1.2}
\end{align}
Now, choose $p\in\left(1,\infty\right)$ and $J\in\mathbb{N}_{0}$
such that $\left(r-1\right)p>1$, $Jp'>1$, and $$2\left(\frac{1}{p'}-\frac{3}{4}\right)=2\left(\frac{1}{4}-\frac{1}{p}\right)>-1.$$  The condition $r>1$ ensures
that this choice is possible.  Using the H\"older inequality in the summation over $l_2$ to bound by the $\ell_{l_2}^p$ and $\ell_{l_2}^{p'}$ norms, we then have
\begin{align*} 
(\ref{eq:a1.2})&\lesssim \sum_{l_{1}}l_{1}^{J-3/4}\left\Vert P_{l_{1}}w\right\Vert _{2}A\frac{1}{k^{r-1-\frac{1}{p}}l_{1}^{J-\frac{1}{p'}}}\\
&=\frac{A}{k^{r-1-\frac{1}{p}}}\sum_{l_{1}}l_{1}^{\frac{1}{p'}-\frac{3}{4}}\left\Vert P_{l_{1}}w\right\Vert _{2}\\
&\lesssim\frac{A}{k^{r-1-1/p}}Bk^{\frac{1}{p'}-\frac{1}{4}},
\end{align*}
where we have used the H\"older inequality again to obtain the last inequality.  This completes the estimate of the $\mathcal{A}_1$ contribution.

To estimate the contribution of $\mathcal{A}_2$, we again subdivide into further cases.  We first consider the contribution from 
$$\mathcal{A}_{2a}=\{k<\left|l_{1}-l_{2}\right|<2k+2\}\cap \mathcal{A}_2.$$
Here, we have $k\lesssim l_{1}\sim l_{2}$, and we invoke the change of variable $l_{2}=l_{1}+j,$ where $\left|j\right|\lesssim k$.
The contribution is bounded by
\begin{align*}
\sum_{j}\sum_{l_{1}}l_{1}^{1/4}\left\Vert P_{l_{1}}w\right\Vert _{2}\left\Vert P_{l_{1}+j}w\right\Vert _{2}&\leq\sum_{j}\sum_{l_{1}}\frac{A}{l_{1}^{r-1/4}}\left\Vert P_{l_{1}+j}w\right\Vert _{2}\\
&\lesssim\sum_{j}\frac{A}{k^{r-3/4}}B\lesssim\frac{AB}{k^{r-7/4}},
\end{align*}
where the last line follows from the H\"older inequality, and where we have used $2\left(r-1/4\right)>1$.

The remaining contribution from $\mathcal{A}_2$ is the contribution of
$$\mathcal{A}_{2b}=\{2k+2\leq\left|l_{1}-l_{2}\right|\}\cap \mathcal{A}_2.$$  Here, we have $k\geq1+\lambda_{1}$,
and thus $\left|\left|l_{1}-l_{2}\right|-2\right|\sim\left|l_{1}-l_{2}\right|$.
Using \Lemref{trilinear_new}, the contribution is bounded by
\begin{align}
\nonumber &\sum_{l_{1}}\sum_{l_{2}}l_{1}^{1/4}\frac{k^{J}}{\left|l_{2}-l_{1}\right|^{J}}\frac{1}{l_{1}}\left\Vert P_{l_{1}}w\right\Vert _{2}l_{2}\left\Vert P_{l_{2}}w\right\Vert _{2}\\
&\hspace{0.2in}\leq Ak^{J}\sum_{l_{1}}\frac{1}{l_{1}^{3/4}}\left\Vert P_{l_{1}}w\right\Vert _{2}\sum_{l_{2}}\frac{1}{\left|l_{2}-l_{1}\right|^{J}}\cdot\frac{1}{l_{2}^{r-1}}\label{eq:a2.2}
\end{align}
Choosing $J$ and $p$ such that $Jp>1,\left(r-1\right)p'>1$ (this is possible, since $r>1$), and using the H\"older inequality to estimate the summation in $l_2$ by appropriate $\ell_{l_2}^p$ and $\ell_{l_2}^{p'}$ norms, we obtain
\begin{align*}
(\ref{eq:a2.2})&\lesssim Ak^{J}\sum_{l_{1}}\frac{1}{l_{1}^{3/4}}\left\Vert P_{l_{1}}w\right\Vert _{2}\frac{1}{k^{J-1/p}}\cdot\frac{1}{k^{r-1-1/p'}}\\
&=A\frac{1}{k^{r-2}}\sum_{l_{1}}\frac{1}{l_{1}^{3/4}}\left\Vert P_{l_{1}}w\right\Vert _{2}\\
&\lesssim A\frac{1}{k^{r-2}}\frac{1}{k^{1/4}}B,
\end{align*}
where the last line follows from another application of the H\"older inequality.

We now address the contribution of $\mathcal{A}_3$.  This further splits into:
\begin{align}
\mathcal{A}_{3a}=\{k+2l_{2}+2>l_{1}\geq k>l_{2}\}\cap\mathcal{A}_3\label{eq:a3.1}
\end{align}
and
\begin{align}
\mathcal{A}_{3b}=\{l_{1}\geq k+2l_{2}+2\geq k>l_{2}\}\cap\mathcal{A}_3.\label{eq:a3.2}
\end{align}
To handle the contribution of (\ref{eq:a3.1}), note that in this case $l_{2}<k\sim l_{1}$, and that for each $l_{2}$, the number of choices of $l_1$ contributing to the sum is $O(l_2)$.  The contribution is thus bounded by
\begin{align*}
&\sum_{l_{2}}\sum_{l_{1}}l_{2}^{1/4}\frac{1}{k}\left\Vert P_{l_{1}}w\right\Vert _{2}l_{2}\left\Vert P_{l_{2}}w\right\Vert _{2}\\
&\hspace{0.2in}\lesssim\sum_{l_{2}}\sum_{l_{1}}\frac{A}{k^{r+1}}l_{2}^{5/4}\left\Vert P_{l_{2}}w\right\Vert _{2}\\
&\hspace{0.2in}\lesssim \frac{A}{k^{r+1}}\sum_{l_{2}}l_{2}^{9/4}\left\Vert P_{l_{2}}w\right\Vert _{2}\\
&\hspace{0.2in}\lesssim \frac{A}{k^{r+1}}k^{11/4}B=\frac{AB}{k^{r-7/4}},
\end{align*}
where, in passing from the second to third lines, we've used the bound on the number of terms in the summation over $l_1$, and in passing to the last line, we've used the H\"older inequality.

We now turn to the contribution of (\ref{eq:a3.2}).  Here $k+2l_{2}+2\sim k$.
Using \Lemref{trilinear_new}, this contribution is bounded by
\begin{align*}
&\sum_{l_{2}}\sum_{l_{1}}l_{2}^{1/4}\frac{l_{2}^{J}}{\left(l_{1}-k-2\right)^{J}}\frac{1}{l_{1}}\left\Vert P_{l_{1}}w\right\Vert _{2}l_{2}\left\Vert P_{l_{2}}w\right\Vert _{2}\\
&\hspace{0.2in}\lesssim\sum_{l_{2}}\sum_{l_{1}}\frac{l_{2}^{J+5/4}}{\left(l_{1}-k-2\right)^{J}}\frac{A}{l_{1}^{r+1}}\left\Vert P_{l_{2}}w\right\Vert _{2}\\
&\hspace{0.2in}\lesssim\sum_{l_{2}}\frac{l_{2}^{J+5/4}}{l_{2}^{J-1/2}}\frac{A}{k^{r+1/2}}\left\Vert P_{l_{2}}w\right\Vert _{2}\\
&\hspace{0.2in}\lesssim k^{9/4}\frac{A}{k^{r+1/2}}B=\frac{AB}{k^{r-7/4}},
\end{align*}
where we have used the H\"older inequality in the last two lines.

We similarly divide $\mathcal{B}=\{l_{1}+l_{2}<k\}$ into smaller regions.  The first of these is $\mathcal{B}_1=\{l_{1}\geq l_{2}\}$, which we subdivide into two
further sets of indices.  The first contribution is that of
$$\mathcal{B}_{1a}=\{k\geq l_{1}+2l_{2}+2;l_{1}\leq\frac{k}{2}\}\cap\mathcal{B}_1.$$ 
Here, because $k\geq 10$, we have $k-l_{1}-2\sim k$, and the 
contribution is bounded by
\begin{align*}
&\sum_{l_{2}}\sum_{l_{1}}l_{2}^{1/4}\frac{l_{2}^{J}}{k^{J}}\frac{1}{l_{1}}\left\Vert P_{l_{1}}w\right\Vert _{2}l_{2}\left\Vert P_{l_{2}}w\right\Vert _{2}\\
&\hspace{0.2in}\leq\frac{1}{k^{J}}\sum_{l_{1}}\sum_{l_{2}}l_{2}^{J+\frac{5}{4}}\frac{1}{l_{1}^{r+1}}A\left\Vert P_{l_{2}}w\right\Vert _{2}\\
&\hspace{0.2in}\lesssim\frac{1}{k^{J}}\sum_{l_{1}}l_{1}^{J+\frac{7}{4}}\frac{1}{l_{1}^{r+1}}AB\\
&\hspace{0.2in}\lesssim\frac{1}{k^{J}}k^{J+\frac{7}{4}-r}AB=\frac{AB}{k^{r-7/4}},
\end{align*}
where we have again used the H\"older inequality to pass from the second to third lines, and where we have used $J+\frac{7}{4}-r-1>-1$ (which holds trivially).

The two remaining subdivisions of the index set $\{l_1\geq l_2\}$ are 
$$\mathcal{B}_{1b}=\{k\geq l_{1}+2l_{2}+2;l_{1}>\frac{k}{2}\}\cap\mathcal{B}_1\quad\textrm{and}\quad \mathcal{B}_{1c}=\{l_{1}+2l_{2}+2>k\}\cap\mathcal{B}_1.$$
In both cases, $l_{1}\sim k$, and the contribution is bounded by
\begin{align*}
&\sum_{l_{2}}\sum_{l_{1}}l_{2}^{1/4}\frac{1}{k}\left\Vert P_{l_{1}}w\right\Vert _{2}l_{2}\left\Vert P_{l_{2}}w\right\Vert _{2}\\
&\hspace{0.2in}\lesssim\frac{A}{k^{r+1}}\sum_{l_{2}}\sum_{l_{1}}l_{2}^{5/4}\left\Vert P_{l_{2}}w\right\Vert _{2}\\
&\hspace{0.2in}\lesssim\frac{A}{k^{r}}\sum_{l_{2}}l_{2}^{5/4}\left\Vert P_{l_{2}}w\right\Vert _{2}\\
&\hspace{0.2in}\lesssim\frac{A}{k^{r}}k^{\frac{7}{4}}B=\frac{AB}{k^{r-7/4}}
\end{align*}
where in passing from the second to third lines we have used that for each $l_2$ there are at most $O(k)$ choices of index $l_1$ contributing to the summation, and in passing to the last line, we've use the H\"older inequality.

The second region contributing to $\mathcal{B}$ is $\mathcal{B}_2=\{l_{1}<l_{2}\}$. This again further splits into several parts.  The first such contribution is
that of $$\mathcal{B}_{2a}=\{k\geq2l_{1}+l_{2}+2;l_{2}\leq\frac{k}{2}\}\cap\mathcal{B}_2.$$ 
Here, because $k\geq 10$, we have  $k-l_{2}-2\sim k$, and the contribution is bounded by
\begin{align*}
&\sum_{l_{2}}\sum_{l_{1}}l_{1}^{1/4}\frac{l_{1}^{J}}{k^{J}}\frac{1}{l_{1}}\left\Vert P_{l_{1}}w\right\Vert _{2}l_{2}\left\Vert P_{l_{2}}w\right\Vert _{2}\\
&\hspace{0.2in}\leq\frac{1}{k^{J}}\sum_{l_{2}}\sum_{l_{1}}l_{1}^{J-\frac{3}{4}}\left\Vert P_{l_{1}}w\right\Vert _{2}\frac{1}{l_{2}^{r-1}}A\\
&\hspace{0.2in}\lesssim\frac{1}{k^{J}}\sum_{l_{2}}l_{2}^{J-\frac{1}{4}}B\frac{1}{l_{2}^{r-1}}A\\
&\hspace{0.2in}\lesssim\frac{1}{k^{J}}k^{J+3/4-r+1}BA=\frac{AB}{k^{r-7/4}}
\end{align*}
where in passing from the second to third lines we have used the H\"older inequality, and where we have used $J+\frac{3}{4}-r>-1$ and $2\left(J-3/4\right)>-1$ (which hold trivially) to bound the sums.

The next contribution comes from
$$\mathcal{B}_{2b}=\{k\geq2l_{1}+l_{2}+2;l_{2}>\frac{k}{2}\}\cap\mathcal{B}_2.$$
For these terms, we have $l_{2}\sim k$, and the contribution is bounded by
\begin{align*}
&\sum_{l_{2}}\sum_{l_{1}}l_{1}^{1/4}\frac{l_{1}^{J}}{\left(k-l_{2}-2\right)^{J}}\frac{1}{l_{1}}\left\Vert P_{l_{1}}w\right\Vert _{2}k\left\Vert P_{l_{2}}w\right\Vert _{2}\\
&\hspace{0.2in}\lesssim\frac{A}{k^{r-1}}\sum_{l_{1}}\sum_{l_{2}}\frac{l_{1}^{J-3/4}}{\left(k-l_{2}-2\right)^{J}}\left\Vert P_{l_{1}}w\right\Vert _{2}\\
&\hspace{0.2in}\lesssim\frac{A}{k^{r-1}}\sum_{l_{1}}\frac{l_{1}^{J-3/4}}{l_{1}^{J-1}}\left\Vert P_{l_{1}}w\right\Vert _{2}\\
&\hspace{0.2in}\lesssim\frac{A}{k^{r-1}}k^{3/4}B=\frac{AB}{k^{r-7/4}},
\end{align*}
where we've again used the H\"older inequality in passing to the last line.

It remains to estimate the contribution from
$$\mathcal{B}_{2c}=\{2l_{1}+l_{2}+2>k\}\cap\mathcal{B}_2.$$ 
Here, $l_{2}\sim k$. 
\begin{align*}
&\sum_{l_{2}}\sum_{l_{1}}l_{1}^{1/4}\frac{1}{l_{1}}\left\Vert P_{l_{1}}w\right\Vert _{2}k\left\Vert P_{l_{2}}w\right\Vert _{2}\\
&\hspace{0.2in}\lesssim\frac{A}{k^{r-1}}\sum_{l_{1}}\sum_{l_{2}}\frac{1}{l_{1}^{3/4}}\left\Vert P_{l_{1}}w\right\Vert _{2}\\
&\hspace{0.2in}\lesssim\frac{A}{k^{r-1}}\sum_{l_{1}}l_{1}^{1/4}\left\Vert P_{l_{1}}w\right\Vert _{2}\\
&\hspace{0.2in}\lesssim\frac{A}{k^{r-1}}k^{3/4}B=\frac{AB}{k^{r-7/4}}
\end{align*}
where in passing from the second line to the third line we have observed that, since $l_{2}<k-l_{1}$, for each fixed $l_{1}$ there are at most $O(l_1)$ choices for the index $l_2$, and in passing to the last line, we've used the H\"older inequality.  

The proof of the claim is now complete.
\end{proof}

\subsection{The harmonic term}

We note that for all $m\in\mathbb{N}_{0}$, $$\left\Vert D^{1}\mathcal{P}_{\mathcal{H}}u\right\Vert _{H^{2m}}\lesssim\left\Vert \mathcal{P}_{\mathcal{H}}u\right\Vert _{H^{2m+1}}\lesssim_{m}\left\Vert \mathcal{P}_{\mathcal{H}}u\right\Vert _{2}\lesssim B.$$ 

We also observe that 
\begin{align*}
\left\Vert D^{1}\mathcal{P}_{\mathcal{H}}u\right\Vert _{H^{2m}} & \sim_{M}\left\Vert \pi_{0}D^{1}\mathcal{P}_{\mathcal{H}}u\right\Vert _{2}+\left\Vert \left(1-\pi_{0}\right)D^{1}\mathcal{P}_{\mathcal{H}}u\right\Vert _{H^{2m}}\\
&\sim\left\Vert \pi_{0}D^{1}\mathcal{P}_{\mathcal{H}}u\right\Vert _{2}+\left\Vert \Delta^{m}D^{1}\mathcal{P}_{\mathcal{H}}u\right\Vert _{2}\\
 & \sim\left\Vert \pi_{0}D^{1}\mathcal{P}_{\mathcal{H}}u\right\Vert _{2}+\left\Vert \left\Vert P_{k}k^{2m}D^{1}\mathcal{P}_{\mathcal{H}}u\right\Vert _{2}\right\Vert _{l_{k}^{2}\left(\mathbb{N}_{0}+\lambda_{1}\right)}.
\end{align*}
As a consequence, for all $k\in\mathbb{N}_{0}+\lambda_{1}$ and $m\in\mathbb{N}_{0}$, we have
$$\left\Vert P_{k}D^{1}\mathcal{P}_{\mathcal{H}}u\right\Vert \lesssim_{M,m}\frac{B}{k^{2m}}.$$
Choosing $m=m\left(r\right)$ large enough then leads to
\[
\left\Vert P_{k}D^{1}\mathcal{P}_{\mathcal{H}}u\right\Vert _{2}\lesssim_{M,r}\frac{AB}{k^{r-7/4}}.
\]

\subsection{The linear terms}

All the remaining terms can be can be summarized by the following
estimate, which can be proved by a stationary phase argument.
\begin{claim}
Let $a,b\in\mathbb{N}_{0}$ such that $a-2b\leq1$. We write $D_{B}^{k}$
as a schematic for a spatial differential operator of order
$k$, such that any local coefficients $c(x)$ of $D_{B}^{k}$ satisfy
\[
\left\Vert c(x)\right\Vert _{C^{m}}\lesssim_{m,\neg B}B
\]
Then for all $k\in\mathbb{N}_{0}+\lambda_{1}+10$,
\[
\sum_{l\in\mathbb{N}_{0}+\lambda_{1}}\left\Vert P_{k}\left(D_{B}^{a}\left(-\Delta\right)^{-b}P_{l}w\right)\right\Vert _{2}\lesssim_{a,b,\neg k}\frac{AB}{k^{r-7/4}}.
\]
\end{claim}

\begin{proof}
This is equivalent to proving that for any $\left(v_{l}\right)_{l\in\mathbb{N}_{0}+\lambda_{1}}$
where $\left\Vert v_{l}\right\Vert _{L^{2}M}\leq1$ for all $l$, we have the bound
\begin{align}
\nonumber &\sum_{l\in\mathbb{N}_{0}+\lambda_{1}}\left|\left\langle \left\langle P_{k}\left(D_{B}^{a}\left(-\Delta\right)^{-b}P_{l}w\right),v_{l}\right\rangle \right\rangle \right|\\
\nonumber &\hspace{1.8in}=\sum_{l\in\mathbb{N}_{0}+\lambda_{1}}\left|\left\langle \left\langle D_{B}^{a}\left(-\Delta\right)^{-b}P_{l}w,P_{k}v_{l}\right\rangle \right\rangle \right|\\
&\hspace{1.8in}\lesssim_{\neg k}\frac{AB}{k^{r-7/4}}.\label{eq:claim15}
\end{align}

To show (\ref{eq:claim15}), fix $\varepsilon\in\left(0,\frac{1}{2}\right)$.  Handling the ``critical
region'' $\left[k-k^{\varepsilon},k+k^{\varepsilon}\right]$ (where $l\sim_{\varepsilon}k$) is simple: 
\begin{align*}
&\sum_{l\in\left[k-k^{\varepsilon},k+k^{\varepsilon}\right]}\left|\left\langle \left\langle D_{B}^{a}\left(-\Delta\right)^{-b}P_{l}w,P_{k}v_{l}\right\rangle \right\rangle \right|\\
&\hspace{1.2in}\lesssim\sum_{l}l^{1/4}l^{a-2b}B\left\Vert P_{l}w\right\Vert _{2}\\
&\hspace{1.2in}\lesssim_{\varepsilon}\sum_{l}\frac{AB}{k^{r-a+2b-\frac{1}{4}}}\\
&\hspace{1.2in}\lesssim\frac{AB}{k^{r-a+2b-\frac{1}{4}-\varepsilon}},
\end{align*}
where the last inequality follows from the observation that there are at most $O(k^\varepsilon)$ choices of $l$ contributing to the summation.
For the region away from $k$, we employ integration by parts to get arbitrary decay, as in the method of stationary phase.

Fix $m=m\left(\varepsilon,r,a\right)\in\mathbb{N}_{1}$ such that
$\varepsilon(m-1)>r+a$. Observe that when $\left|l-k\right|>k^{\varepsilon}$,
we can employ the Fourier trick:
\begin{align}
\nonumber &\left\langle \left\langle D_{B}^{a}\left(-\Delta\right)^{-b}P_{l}w,P_{k}v_{l}\right\rangle \right\rangle \\
\nonumber &\hspace{0.2in}=\sum\limits _{\substack{z_{1}\in[0,1)\cap\left(\sigma\left(\sqrt{-\Delta}\right)-l\right)\\
z_{2}\in[0,1)\cap\left(\sigma\left(\sqrt{-\Delta}\right)-k\right)
}
}\frac{1}{\left(l+z_{1}\right)^{2b}}\left\langle \left\langle D_{B}^{a}\pi_{l+z_{1}}w,\pi_{k+z_{2}}v_{l}\right\rangle \right\rangle \\
\nonumber &\hspace{0.2in}=\sum_{z_{1},z_{2}}\frac{1}{\left(l+z_{1}\right)^{2b}}\left\langle \left\langle \pi_{l+z_{1}}w,D_{B}^{a}\pi_{k+z_{2}}v_{l}\right\rangle \right\rangle\\
&\hspace{0.2in}=\sum_{z_{1},z_{2}}\frac{1}{\left(l+z_{1}\right)^{2b}}\cdot\frac{1}{\left(k+z_{2}\right)^{2}-\left(l+z_{1}\right)^{2}}\left\langle \left\langle \pi_{l+z_{1}}w,\left[D_{B}^{a},-\Delta\right]\pi_{k+z_{2}}v_{l}\right\rangle \right\rangle.\label{eq:fourier}
\end{align}
An induction argument now shows that 
\begin{align*}
(\ref{eq:fourier})&=\sum_{z_{1},z_{2}}\frac{1}{\left(\left(l+z_{1}\right)^{2}\right)^{b}}\cdot\frac{1}{\left(\left(k+z_{2}\right)^{2}-\left(l+z_{1}\right)^{2}\right)^{m}}\left\langle \left\langle \pi_{l+z_{1}}w,D_{B}^{a+m}\pi_{k+z_{2}}v_{l}\right\rangle \right\rangle.
\end{align*}
As in the proof of \Lemref{bilinear_1_new}, we let $\Psi\left(z_{1},z_{2}\right)=\frac{1}{\left(\left(l+z_{1}\right)^{2}\right)^{b}}\cdot\frac{1}{\left(\left(k+z_{2}\right)^{2}-\left(l+z_{1}\right)^{2}\right)^{m}}$
and observe that 
\[
\left\Vert \Psi\right\Vert _{C^{2}\left(\left[0,1\right]^{2}\right)}\lesssim_{\lambda_{1}}\frac{1}{l^{2b}}\cdot\frac{1}{\left(k^{2}-l^{2}\right)^{m}}\lesssim_{\lambda_{1},b}\frac{1}{\left(k^{2}-l^{2}\right)^{m}}
\]
So by the Fourier trick and \Lemref{bilinear_1_new}, we conclude
\begin{align*}
&\left|\left\langle \left\langle D_{B}^{a}\left(-\Delta\right)^{-b}P_{l}w,P_{k}v_{l}\right\rangle \right\rangle \right|\\
&\hspace{0.2in}\lesssim_{M,m,\neg l,\neg k}\frac{1}{\left(k^{2}-l^{2}\right)^{m}}k^{\frac{1}{4}}\left\Vert P_{l}\omega\right\Vert _{2}k^{m+a}\left\Vert P_{k}v_{l}\right\Vert _{2}B\\
&\hspace{0.2in}\lesssim_{\lambda_{1}}\frac{k^{m+a+1/4}}{\left(k^{2}-l^{2}\right)^{m}}AB.
\end{align*}
We observe that, setting $z:=l-k$,
\begin{align*}
\sum_{l\notin\left[k-k^{\varepsilon},k+k^{\varepsilon}\right]}\frac{k^{m+a+1/4}}{\left(k^{2}-l^{2}\right)^{m}}&\lesssim\int_{|z|>k^{\varepsilon}}\frac{k^{m+a+1/4}}{\left|2zk+z^{2}\right|^{m}}\;\mathrm{d}z\\
&\lesssim\int_{|z|>k^{\varepsilon}}\frac{k^{a+1/4}}{\left|z\right|^{m}}\;\mathrm{d}z\\
&\lesssim_{m}\frac{k^{a+1/4}}{k^{\varepsilon\left(m-1\right)}}.
\end{align*}
Because of the way we picked $m$, we conclude 
\[
\sum_{l\notin\left[k-k^{\varepsilon},k+k^{\varepsilon}\right]}\left|\left\langle \left\langle D^{a}\left(-\Delta\right)^{-b}P_{l}w,P_{k}v_{l}\right\rangle \right\rangle \right|\lesssim_{M,r,a,b,\neg k}\frac{AB}{k^{r-7/4}},
\]
which completes the proof of the claim.
\end{proof}

\appendix

\section{Review of Differential geometry}
\label{appendix:notation}

In this appendix we recall our conventions for some standard notation from differential geometry which we use throughout
the paper.  For any tensor $T_{a_{1}...a_{k}}$, $\left(\nabla T\right)_{ia_{1}...a_{k}}=\nabla_{i}T_{a_{1}...a_{k}}$
and $\Div_{g}T=\nabla^{i}T_{ia_{2}...a_{k}}.$

Moreover,
$$\left(d\omega\right)_{ba_{1}...a_{k}}=\left(k+1\right)\widetilde{\nabla}_{[b}\omega_{a_{1}...a_{k}]}\;\forall\omega\in\Omega^{k}(M),$$
where $\widetilde{\nabla}$ is any torsion-free connection,
$$\left(\delta\omega\right)_{a_{1}...a_{k-1}}=-\nabla^{b}\omega_{ba_{1}...a_{k-1}}=-(\Div_{g}w)_{a_{1}...a_{k-1}}\forall\omega\in\Omega^{k}(M),$$
and
$$\left(\nabla_{a}\nabla_{b}-\nabla_{b}\nabla_{a}\right)T^{ij}{}_{kl}=-R_{ab\sigma}{}^{i}T^{\sigma j}{}_{kl}-R_{ab\sigma}{}^{j}T^{i\sigma}{}_{kl}+R_{abk}{}^{\sigma}T^{ij}{}_{\sigma l}+R_{abl}{}^{\sigma}T^{ij}{}_{k\sigma},$$
for any tensor $T^{ij}{}_{kl}$, where $R$ is the Riemann
curvature tensor and $\nabla$ the Levi-Civita connection. 

Similar
identities hold for other types of tensors. When we do not care about
the exact indices and how they contract, we can just write the schematic
identity $\left(\nabla_{a}\nabla_{b}-\nabla_{b}\nabla_{a}\right)T^{ij}{}_{kl}=R*T.$
As $R$ is bounded on compact $M$, interchanging derivatives is a
zeroth-order operation on $M$.

For any tensor field $T^{a_{1}...a_{k}}{}_{b_{1}...b_{l}}$
(and vector field $X$), the Lie derivative is given by 
\begin{align*}
\left(\mathcal{L}_{X}T\right)^{a_{1}...a_{k}}{}_{b_{1}...b_{l}}= & X^{c}\nabla_{c}T^{a_{1}...a_{k}}{}_{b_{1}...b_{l}}-\Sigma_{i=1}^{k}T^{a_{1}...c...a_{k}}{}_{b_{1}...b_{l}}\nabla_{c}X^{a_{i}}\\
 & +\Sigma_{i=1}^{k}T^{a_{1}...a_{k}}{}_{b_{1}...c...b_{l}}\nabla_{b_{i}}X^{c}
\end{align*}
Then we have $\mathcal{L}_{X}\left(A\otimes B\right)=\mathcal{L}_{X}A\otimes B+A\otimes\mathcal{L}_{X}B$
for any tensor fields $A,B$. 

Because $\nabla$ is metric and torsion-free, we have
$$\mathcal{L}_{X}\left\langle Y,Z\right\rangle =\left\langle \nabla_{X}Y,Z\right\rangle +\left\langle Y,\nabla_{X}Z\right\rangle,$$
and
$$\nabla_{X}Y-\nabla_{Y}X=\left[X,Y\right]=\mathcal{L}_{X}Y.$$

We also have
$$d^{2}=0,\quad d\Delta_{H}=\Delta_{H}d,\quad \Delta_{H}\star=\star\Delta_{H},\quad \mathcal{L}_{X}d=d\mathcal{L}_{X},$$
as well as
$$\mathcal{L}_{X}\vol=\Div X\vol,\quad \star1=\vol,\star\vol=1,$$
and
$$d\star=\left(-1\right)^{k}\star\delta,$$ $$\delta\star=(-1)^{k+1}\star d,$$
$$\star\star=\left(-1\right)^{k(2-k)}$$ on $\Omega^{k}\left(M\right)$.

For tensor $T_{a_{1}...a_{k}}$, define the Weitzenbock curvature operator by writing
\begin{align*}
\Ric(T)_{a_{1}...a_{k}}&=2\sum_{j=1}^{k}\nabla_{[i}\nabla_{a_{j}]}T_{a_{1}...a_{j-1}}{}^{i}{}_{a_{j+1}...a_{k}}\\
&=\sum_{j}R_{a_{j}}{}^{\sigma}T_{a_{1}...a_{j-1}\sigma a_{j+1}...a_{k}}-\sum_{j\neq l}R_{a_{j}}{}^{\mu}{}_{a_{l}}{}^{\sigma}T_{a_{1}...\sigma...\mu...a_{k}}
\end{align*}
where $R_{ab}=R_{a\sigma b}{}^{\sigma}$ is the Ricci tensor.
Then we have the Weitzenbock formula,
\[
\Delta_{H}\omega=\nabla_{i}\nabla^{i}\omega-\Ric(\omega)
\]
for all $\omega\in \Omega^k(M)$, where $\nabla_{i}\nabla^{i}\omega=\tr(\nabla^{2}\omega)$ is also
called the {\it connection Laplacian}, which differs from the Hodge
Laplacian by a zeroth-order term. The geometry of $M$ and differential
forms are more easily handled by the Hodge Laplacian, while the connection
Laplacian is more useful in calculations with tensors and the Penrose
notation.

For tensors $T_{a_{1}...a_{k}}$ and $Q_{a_{1}...a_{k}}$, the tensor
inner product is given by $$\left\langle T,Q\right\rangle =T_{a_{1}...a_{k}}Q^{a_{1}...a_{k}}.$$
However, for $\omega,\eta\in\Omega^{k}(M),$ there is another dot product,
called the Hodge inner product, where 
\[
\left\langle \omega,\eta\right\rangle _{\Lambda}=\frac{1}{k!}\left\langle \omega,\eta\right\rangle 
\]
 So $\left|\omega\right|_{\Lambda}=\sqrt{\frac{1}{k!}}\left|\omega\right|.$
We then define $$\left\langle \left\langle \omega,\eta\right\rangle \right\rangle =\int_{M}\left\langle \omega,\eta\right\rangle \vol$$
and $$\left\langle \left\langle \omega,\eta\right\rangle \right\rangle _{\Lambda}=\int_{M}\left\langle \omega,\eta\right\rangle _{\Lambda}\vol.$$
Recall that $\omega\wedge\star\eta=\left\langle \omega,\eta\right\rangle _{\Lambda}\vol$
for all $\omega,\eta\in\Omega^{k}(M)$. Also, for all $\omega\in \Omega^{k}(M)$ and $\eta\in \Omega^{k+1}(M)$, we have
\[
\left\langle \left\langle d\omega,\eta\right\rangle \right\rangle _{\Lambda}=\left\langle \left\langle \omega,\delta\eta\right\rangle \right\rangle _{\Lambda}.
\]

Lastly,
$$\nabla_{X}(\star\omega)=\star\left(\nabla_{X}\omega\right),$$
and $$\left|\star\omega\right|_{\Lambda}=\left|\omega\right|_{\Lambda}$$
for any $\omega\in\Omega^{k}(M),X\in\mathfrak{X}(M)$.

We remark that the signs of $\Ric$ and $\Delta_{H}$ in the literature can differ according to various conventions commonly in use.

\section{Trilinear estimate}
\label{appendix:lemma}

In this appendix, we give the proof of \Lemref{trilinear_new}.  The arguments extend and generalize
the proof of related results in \cite{haniGlobalWellposednessCubic2011}.  We sketch the details
for completeness.  We begin with an integration-by-parts lemma.
\begin{lem}
\label{lem:a.1}
For $i=1,2,3,4$, let $e_{i}\in C^{\infty}\left(M\right)$ be eigenfunctions
where $\left(-\Delta\right)e_{i}=n_{i}^{2}e_{i}$, and assume 
$n_{1}\geq n_{2}\geq n_{3}\geq n_{4}\geq0$ and $n_{1}^{2}\neq n_{2}^{2}+n_{3}^{2}+n_{4}^{2}$.

Set $\mathcal{N}=\frac{1}{n_{1}^{2}-n_{2}^{2}-n_{3}^{2}-n_{4}^{2}}$.  Then, for any $a_{1},a_{2},a_{3},a_{4}\in\mathbb{N}_{0}$ and $m\in\mathbb{N}_{1}$,
we have the schematic identity
\begin{align*}
&\int_{M}\left(\nabla^{a_{1}}e_{2}\right)*\left(\nabla^{a_{2}}e_{2}\right)*\left(\nabla^{a_{3}}e_{3}\right)*\left(\nabla^{a_{4}}e_{4}\right)\\
&\hspace{0.2in}= \mathcal{N}^{m}\sum_{\substack{b_{2}+b_{3}+b_{4}=2m\\
0\leq b_{2},b_{3},b_{4}\leq m
}
}\int_{M}\nabla^{a_{1}}e_{1}*\nabla^{a_{2}+b_{2}}e_{2}*\nabla^{a_{3}+b_{3}}e_{3}*\nabla^{a_{4}+b_{4}}e_{4}\\
&\hspace{0.4in} +\mathcal{N}^{m}\sum_{\substack{\sum_{j}c_{j}\leq\sum_{j}a_{j}+2m-2\\
0\leq c_{j}\leq a_{j}+m-1\;\forall j\neq1\\
c_{1}\leq a_{1}
}
}\int_{M}T_{mc_{1}c_{2}c_{3}c_{4}}*\nabla^{c_{1}}e_{1}*\nabla^{c_{2}}e_{2}*\nabla^{c_{3}}e_{3}*\nabla^{c_{4}}e_{4}
\end{align*}
for some smooth tensors $T_{mc_{1}c_{2}c_{3}c_{4}}$. We note that
besides $\mathcal{N}$, there is no dependence on any $n_{i}$. 
\end{lem}

\begin{proof}
Recall that $\Delta f=\nabla_{\alpha}\nabla^{\alpha}f$ for any function
$f$ ($\alpha$ is an abstract index, not a natural number).  Also recall 
that for any tensor $T$, $\nabla_{\alpha}\nabla^{\alpha}*\nabla^{k}T=\nabla^{k}*\nabla_{\alpha}\nabla^{\alpha}T+\nabla^{k}\left(R*T\right),$
where $R$ is the Riemann curvature tensor and $\nabla^{k}\left(R*T\right)=\sum_{i=0}^{k}\nabla^{i}R*\nabla^{k-i}T$.

We then observe that 
\begin{align*}
&n_{1}^{2}\int_{M}\nabla^{a_{1}}e_{1}*\nabla^{a_{2}}e_{2}*\nabla^{a_{3}}e_{3}*\nabla^{a_{4}}e_{4}\\
&\hspace{0.2in}=\int_{M}\nabla^{a_{1}}\left(-\Delta\right)e_{1}*\nabla^{a_{2}}e_{2}*\nabla^{a_{3}}e_{3}*\nabla^{a_{4}}e_{4}\\
&\hspace{0.2in}=\int_{M}\nabla^{a_{1}}e_{1}*\left(-\nabla_{\alpha}\nabla^{\alpha}\right)\left(\nabla^{a_{2}}e_{2}*\nabla^{a_{3}}e_{3}*\nabla^{a_{4}}e_{4}\right)\\
&\hspace{0.4in}+\int_{M}\nabla^{a_{1}}\left(R*e_{1}\right)*\nabla^{a_{2}}e_{2}*\nabla^{a_{3}}e_{3}*\nabla^{a_{4}}e_{4}\\
&\hspace{0.2in}=\left(n_{2}^{2}+n_{3}^{2}+n_{4}^{2}\right)\int_{M}\nabla^{a_{1}}e_{1}*\nabla^{a_{2}}e_{2}*\nabla^{a_{3}}e_{3}*\nabla^{a_{4}}e_{4}\\
&\hspace{0.4in}+\sum_{\substack{b_{2}+b_{3}+b_{4}=2\\
b_{2},b_{3},b_{4}\leq1
}
}\int_{M}\nabla^{a_{1}}e_{1}*\nabla^{a_{2}+b_{2}}e_{2}*\nabla^{a_{3}+b_{3}}e_{3}*\nabla^{a_{4}+b_{4}}e_{4}\\
&\hspace{0.4in}+\sum_{\substack{c_{j}\leq a_{j}\;\forall j}
}\int_{M}T_{c_{1}c_{2}c_{3}c_{4}}*\nabla^{c_{1}}e_{1}*\nabla^{c_{2}}e_{2}*\nabla^{c_{3}}e_{3}*\nabla^{c_{4}}e_{4}
\end{align*}

This yields
\begin{align}
\nonumber&\int_{M}\nabla^{a_{1}}e_{1}*\nabla^{a_{2}}e_{2}*\nabla^{a_{3}}e_{3}*\nabla^{a_{4}}e_{4}\\
\nonumber&\hspace{0.2in}=\mathcal{N}\sum_{\substack{b_{2}+b_{3}+b_{4}=2\\
b_{2},b_{3},b_{4}\leq1
}
}\int_{M}\nabla^{a_{1}}e_{1}*\nabla^{a_{2}+b_{2}}e_{2}*\nabla^{a_{3}+b_{3}}e_{3}*\nabla^{a_{4}+b_{4}}e_{4}\\
&\hspace{0.4in}+\mathcal{N}\sum_{\substack{c_{j}\leq a_{j}\;\forall j}
}\int_{M}T_{c_{1}c_{2}c_{3}c_{4}}*\nabla^{c_{1}}e_{1}*\nabla^{c_{2}}e_{2}*\nabla^{c_{3}}e_{3}*\nabla^{c_{4}}e_{4}\label{eq:Appendix_Eq1}
\end{align}

On the other hand, for any smooth tensor $T$, we have
\begin{align*}
&n_{1}^{2}\int_{M}T*\nabla^{a_{1}}e_{1}*\nabla^{a_{2}}e_{2}*\nabla^{a_{3}}e_{3}*\nabla^{a_{4}}e_{4}\\
&\hspace{0.2in}=\int_{M}T*\nabla^{a_{1}}\left(-\Delta\right)e_{1}*\nabla^{a_{2}}e_{2}*\nabla^{a_{3}}e_{3}*\nabla^{a_{4}}e_{4}\\
&\hspace{0.2in}=\left(n_{2}^{2}+n_{3}^{2}+n_{4}^{2}\right)\int_{M}T*\nabla^{c_{1}}e_{1}*\nabla^{c_{2}}e_{2}*\nabla^{c_{3}}e_{3}*\nabla^{c_{4}}e_{4}\\
&\hspace{0.4in}+\sum_{\substack{\sum_{j}c_{j}\leq\sum_{j}a_{j}+2\\
c_{j}\leq a_{j}+1\;\forall j\neq1\\
c_{1}\leq a_{1}
}
}\int_{M}T_{c_{1}c_{2}c_{3}c_{4}}*\nabla^{c_{1}}e_{1}*\nabla^{c_{2}}e_{2}*\nabla^{c_{3}}e_{3}*\nabla^{c_{4}}e_{4},
\end{align*}
which gives
\begin{align}
\nonumber &\int_{M}T*\nabla^{a_{1}}e_{1}*\nabla^{a_{2}}e_{2}*\nabla^{a_{3}}e_{3}*\nabla^{a_{4}}e_{4}\\
&\hspace{0.2in}=\mathcal{N}\sum_{\substack{\sum_{j}c_{j}\leq\sum_{j}a_{j}+2\\
c_{j}\leq a_{j}+1\;\forall j\neq1\\
c_{1}\leq a_{1}
}
}\int_{M}T_{c_{1}c_{2}c_{3}c_{4}}*\nabla^{c_{1}}e_{1}*\nabla^{c_{2}}e_{2}*\nabla^{c_{3}}e_{3}*\nabla^{c_{4}}e_{4}.\label{eq:Appendix_Eq2}
\end{align}
Fix $a_1,a_2,a_3,a_4$. We now use induction. To simplify notation, we write $A(s,t)$ for
\[
\sum_{\substack{c_{2}+c_{3}+c_{4}=s\\
\max\left(c_{2}-a_{2},c_{3}-a_{3},c_{4}-a_{4}\right)\leq t\\
c_{2}\geq a_{2},c_{3}\geq a_{3},c_{4}\geq a_{4}
}
}\int_{M}\nabla^{a_{1}}e_{1}*\nabla^{c_{2}}e_{2}*\nabla^{c_{3}}e_{3}*\nabla^{c_{4}}e_{4}.
\]
Similarly, we write $B(s,t)$ for any linear combination of terms $\int_{M}T*\nabla^{c_{1}}e_{1}*\nabla^{c_{2}}e_{2}*\nabla^{c_{3}}e_{3}*\nabla^{c_{4}}e_{4}$
where $s\geq c_{1}+c_{2}+c_{3}+c_{4}-a_{1}$, $t\geq\max\left(c_{2}-a_{2},c_{3}-a_{3},c_{4}-a_{4}\right)$,
$c_{1}\leq a_{1}$ and $T$ is a smooth tensor.

Then (\ref{eq:Appendix_Eq1}) implies $A(s,t)=\mathcal{N}A(s+2,t+1)+\mathcal{N}B(s,t)$,
while (\ref{eq:Appendix_Eq2}) implies $B(s,t)=\mathcal{N}B(s+2,t+1)$. A straightforward
induction argument then gives 
\[
A(s,0)=\mathcal{N}^{m}A(s+2m,m)+\mathcal{N}^{m}B(s+2m-2,m-1),
\]
which was the desired claim.
\end{proof}

\begin{rem}
Lemma \ref{lem:a.1} generalizes to an arbitrary number of functions.
In fact, we only need the case of three functions (making $e_{4}=1$).
In this case, the first term on the right hand side naturally simplifies
to $\mathcal{N}^{m}\int_{M}\nabla^{a_{1}}e_{1}*\nabla^{a_{2}+m}e_{2}*\nabla^{a_{3}+m}e_{3}$.
\end{rem}

We are ready to prove \Lemref{trilinear_new}. 

\begin{proof}[Proof of \Lemref{trilinear_new}] Let $f_{1},f_{2},f_{3}\in L^{2}\left(M\right)$;
$a_{1},b_{1},a_{2},b_{2},a_{3},b_{3},J\in\mathbb{N}_{0}$ and $l_{1}\geq l_{2}\geq l_{3}\geq\lambda_{1}(M)$ 
be such that $l_{1}=l_{2}+Kl_{3}+2$ for $K>1$.

We pass to eigenspace projections, obtaining
\begin{align}
\nonumber & \int_{M}\left(\nabla^{a_{1}}\left(-\Delta\right)^{-b_{1}}P_{l_{1}}f_{1}\right)*\left(\nabla^{a_{2}}\left(-\Delta\right)^{-b_{2}}P_{l_{2}}f_{2}\right)*\left(\nabla^{a_{3}}\left(-\Delta\right)^{-b_{3}}P_{l_{3}}f_{3}\right)\\
\nonumber &\hspace{0.2in}=\sum\limits _{\substack{z_{j}\in[0,1)\cap\left(\sigma\left(\sqrt{-\Delta}\right)-l_{j}\right)\\
j=1,2,3
}
}\bigg[\left(\prod_{i=1,2,3}\frac{1}{\left(l_{i}+z_{i}\right)^{2b_{i}}}\right)\\
&\hspace{1.2in}\cdot\left(\int_{M} \nabla^{a_{1}}\pi_{l_{1}+z_{1}}f_{1}*\nabla^{a_{2}}\pi_{l_{2}+z_{2}}f_{2}*\nabla^{a_{3}}\pi_{l_{3}+z_{3}}f_{3}\right)\bigg].\label{eq:a.1}
\end{align}

Invoking \Lemref{a.1}, 
and setting
$$\Psi\left(z_{1},z_{2},z_{3}\right):=\left(\prod_{i=1,2,3}\frac{1}{\left(l_{i}+z_{i}\right)^{2b_{i}}}\right)\frac{1}{\left(\left(l_{1}+z_{1}\right)^{2}-\left(l_{2}+z_{2}\right)^{2}-\left(l_{3}+z_{3}\right)^{2}\right)^{J}},$$
we see that the right-hand side of (\ref{eq:a.1}) is equal to
\begin{align*}
& \sum\limits _{\substack{z_{j}\in[0,1)\cap\left(\sigma\left(\sqrt{-\Delta}\right)-l_{j}\right)\\
j=1,2,3
}
} \Psi(z_1,z_2,z_3)\\
&\hspace{0.2in}\cdot\bigg(\int_{M}\nabla^{a_{1}}\pi_{l_{1}+z_{1}}f_{1}*\nabla^{a_{2}+J}\pi_{l_{2}+z_{2}}f_{2}*\nabla^{a_{3}+J}\pi_{l_{3}+z_{3}}f_{3}\\
&\hspace{0.3in}+\sum_{(c_1,c_2,c_3)\in \Xi}\int_{M}T_{Jc_{1}c_{2}c_{3}}*\nabla^{c_{1}}\pi_{l_{1}+z_{1}}f_{1}*\nabla^{c_{2}}\pi_{l_{2}+z_{2}}f_{2}*\nabla^{c_{3}}\pi_{l_{3}+z_{3}}f_{3}\bigg),
\end{align*}
where
\begin{align*}
\Xi:=\bigg\{(c_1,c_2,c_3):&\sum_{j}c_{j}\leq\sum_{j}a_{j}+2J-2,\\
&\,\, 0\leq c_{j}\leq a_{j}+J-1\;\forall j\neq1,\,\,\textrm{and}\,\, c_{1}\leq a_{1}\bigg\}.
\end{align*}

Now, note that 
\begin{align*}
l_{1}^{2}-\left(l_{2}+1\right)^{2}-\left(l_{3}+1\right)^{2}&=\left(K^{2}-1\right)l_{3}^{2}+2+2Kl_{2}l_{3}+2l_{2}+\left(4K-2\right)l_{3}\\
&\gtrsim\max_{i=1,2,3}\left(l_{i}+1\right).
\end{align*} 

As a consequence,
\[
\left\Vert \Psi\right\Vert _{C^{2}\left([0,1]^{3}\right)}\lesssim_{\lambda_{1}}\left(\prod_{i=1,2,3}\frac{1}{l_{i}^{2b_{i}}}\right)\frac{1}{\left(Kl_{2}l_{3}\right)^{J}}.
\]
By the Fourier trick and Proposition \Blackboxref{bilinear_1_old}, we bound
the right-hand side of (\ref{eq:a.1}) by 
\[
\left(\prod_{i=1,2,3}\frac{1}{l_{i}^{2b_{i}}}\right)\frac{1}{\left(Kl_{2}l_{3}\right)^{J}}l_{1}^{a_{1}}\left\Vert P_{l_{1}}f_{1}\right\Vert _{2}l_{3}^{1/4}l_{2}^{a_{2}+J}\left\Vert P_{l_{2}}f_{2}\right\Vert _{2}l_{3}^{a_{3}+J}\left\Vert P_{l_{3}}f_{3}\right\Vert _{2},
\]
where we have used the fact that 
\begin{align*}
\left\Vert \nabla^{a_{1}}P_{l_{1}}f_{1}\right\Vert _{2}&\lesssim\left\Vert P_{l_{1}}f_{1}\right\Vert _{H^{a_{1}}}\\
&\sim\left\Vert \left(-\Delta\right)^{a_{1}/2}P_{l_{1}}f_{1}\right\Vert _{2}\\
&\sim_{M}l_{1}^{a_{1}}\left\Vert P_{l_{1}}f_{1}\right\Vert _{2}.
\end{align*}
This completes the proof of the lemma.
\end{proof}


\begin{thebibliography}{99}
\bibitem
{arnoldGeometrieDifferentielleGroupes1966} V. Arnold.  Sur la G\'eometrie Diff\'erentielle des Groupes de Lie de dimension infinie et ses applications \'a l'hydrodynamique des fluides parfaits.  Ann. Inst. Fourier 16 (1966), no. 1, 319--361.
\bibitem
{burqMultilinearEigenfunctionEstimates2005} N. Burq, P. G\'erard and N. Tzvetkov.  Multilinear eigenfunction estimates and global existence for the three dimensional nonlinear Schr\"odinger equations.  Ann. Sci. \'Ecole Norm. Sup. (4) 38 (2005), no. 2, 255--301.
\bibitem
{chanFormulationNavierStokesEquations2017} C.H. Chan, M. Czubak and M. Disconzi.  The formulation of the Navier-Stokes equations on Riemannian manifolds.  J. Geom. Phys. 121 (2017), 335--346.
\bibitem
{christRegularityInversesSingular1988} M. Christ.  On the regularity of inverses of singular integral operators.  Duke Math. J. 57 (1988), no. 2, 459--484.
\bibitem
{caoNavierStokesEquationsRotating1999} C. Cao, M.A. Rammaha, and E.S. Titi.  The Navier-Stokes equations on the rotating 2D sphere: Gevrey regularity and asymptotic degrees of freedom.  Z. Angew. Math. Phys. 50 (1999), no. 3, 341--360.
\bibitem
{caoRammahaTiti2000} C. Cao, M.A. Rammaha, and E.S. Titi.  Gevrey Regularity for Nonlinear Analytic Parabolic Equations on the Sphere.  J. Dyn. Diff. Eq. 12 (2000), no. 2, 411--433.
\bibitem
{ebinGroupsDiffeomorphismsMotion1970} D.G. Ebin and J. Marsden.  Groups of Diffeomorphisms and the Motion of an Incompressible Fluid.  Ann. of Math 92 (1970), no. 1, 102--163.
\bibitem
{fujitaNavierStokesInitialValue1964} H. Fujita and T. Kato.  On the Navier-Stokes initial value problem, I.  Arch. Ration. Mech. Anal. 16 (1964), no. 4, 269--315.
\bibitem
{haniGlobalWellposednessCubic2011} Z. Hani.  Global well-posedness of the cubic nonlinear Schr\"odinger equation on closed manifolds.  Comm. Partial Differential Equations 37 (2012), no. 7, 1186--1236.
\bibitem
{haniPersonalCommunication} Z. Hani.  Personal communication.
\bibitem
{huynhHodgetheoreticAnalysisManifolds2019} K.M. Huynh.  Hodge-Theoretic Analysis on Manifolds with Boundary, Heatable Currents, and Onsager's Conjecture in Fluid Dynamics.  Preprint (2019), arXiv:1907.05360.
\bibitem
{kobayashiNavierStokesEquations2008} M.H. Kobayashi.  On the Navier-Stokes equations on manifolds with curvature.  J. Eng. Math. 60 (2008), no. 1, 55--68.
\bibitem
{ladyzhenskaya} O. Ladyzhenskaya, The Mathematical Theory of Viscous Incompressible Flows, 2nd edition, Gordon and Breach, New York, 1969.
\bibitem
{leeManifoldsDifferentialGeometry2009} J.M. Lee.  Manifolds and Differential Geometry.  Grad. Studies in Math. v. 107, Amer. Math. Soc., Providence, RI, 2009.
\bibitem
{mattinglyElementaryProofExistence1999} J.C. Mattingly and Ya G. Sinai.  An Elementary Proof of the Existence and Uniqueness Theorem for the Navier-Stokes Equations.  Commun. Contemp. Math. 1 (1999), no. 4, 497--516.
\bibitem
{mitreaNavierStokesEquationsLipschitz2001} M. Mitrea and M. Taylor.  Navier-Stokes equations on Lipschitz domains in Riemannian manifolds.  Math. Ann 321 (2001), no. 4, 955--987.
\bibitem
{pruessNavierStokesEquationsSurfaces2020} J. Pruess, G. Simonett, and M. Wilke.  On the Navier-Stokes equations on surfaces.  Preprint (2020), arXiv:2005.00830.
\bibitem
{schwarzHodgeDecompositionMethod1995} G. Schwarz.  Hodge Decomposition--A Method for Solving Boundary Value Problems.  Lect. Notes in Math. Vol. 1607, Springer, Berlin, Heidelberg, 1995.
\bibitem
{taylorPartialDifferentialEquations2011a} M.E. Taylor.  Partial Differential Equations I: Basic Theory.  Applied Math. Sciences, Vol. 115, Springer New York, New York, NY, 2011.
\bibitem
{taylorPartialDifferentialEquations2011} M.E. Taylor.  Partial Differential Equations III: Nonlinear Equations.  Applied Math. Sciences, Vol. 117, Springer New York, New York, NY, 2011.
\bibitem
{waldGeneralRelativity1984} R.M. Wald.  General Relativity.  Univ. of Chicago Press, Chicago, IL, 1984.
\end{thebibliography}
\end{document}